\documentclass[reqno, 12pt]{amsart}
\usepackage{amsmath,amssymb,mathrsfs,amsthm,amsfonts,mathtools}
\usepackage[inline]{enumitem} 
\usepackage[usenames,dvipsnames]{xcolor}
\usepackage{graphicx} 
\usepackage[paper=letterpaper,margin=1in]{geometry}

\definecolor{hypercolor}{HTML}{003399}
\usepackage{hyperref}
 	
\hypersetup{colorlinks=true, linkcolor=hypercolor,citecolor=hypercolor}

\usepackage{newtxtext}

\numberwithin{equation}{section}
\allowdisplaybreaks

\newtheorem{thm}{Theorem}[section]
\newtheorem{lem}[thm]{Lemma}
\newtheorem{prop}[thm]{Proposition}
\newtheorem{cor}[thm]{Corollary}
\theoremstyle{definition}

\theoremstyle{remark}

\newcommand{\sg}{\mathcal{Q}}					
\newcommand{\sgsum}{\mathcal{R}}				
\newcommand{\sgg}{\mathcal{W}}					
\newcommand{\heatsg}{\mathcal{P}}				
\newcommand{\Jop}{\mathcal{J}}					
\newcommand{\jfn}{\mathsf{j}}
\newcommand{\Top}{\mathcal{T}}					

\newcommand{\Aop}{\mathcal{A}}
\newcommand{\Bop}{\mathcal{B}}
\newcommand{\Cop}{\mathcal{C}}
\newcommand{\Gop}{\mathcal{G}}
\newcommand{\Hop}{\mathcal{H}}

\newcommand{\pair}{\mathrm{Pair}}
\newcommand{\dgm}{\mathrm{Dgm}}

\newcommand{\textc}{\mathrm{c}}

\newcommand{\yc}{y_\mathrm{c}}

\newcommand{\vecu}{\vec{u}}
\newcommand{\vecv}{\vec{v}}
\newcommand{\vecalpha}{\vec{\alpha}}
\newcommand{\vecbeta}{\vec{\beta}}

\newcommand{\scrI}{\mathscr{I}}

\newcommand{\scrK}{\mathscr{K}}
\newcommand{\calI}{\mathcal{I}}
\newcommand{\calJ}{\mathcal{J}}
\newcommand{\calK}{\mathcal{K}}
\newcommand{\calL}{\mathcal{L}}

\newcommand{\type}{\operatorname{type}}
\newcommand{\hk}{p}								
\newcommand{\hkk}{p'}							
\newcommand{\logf}{\mathrm{Log}}				
\newcommand{\compl}{\mathrm{c}}					
\newcommand{\Wevent}{\Omega}
\renewcommand{\emptyset}{\varnothing}

\newcommand{\gmc}{\mathcal{G}}
\newcommand{\pspace}{\mathfrak{X}}				

\newcommand{\base}{\mu}							
\newcommand{\kernel}{\mathsf{k}}				
\newcommand{\noise}{\xi}
\newcommand{\Pathsp}{\Gamma}
\newcommand{\ppath}{\gamma}
\newcommand{\PM}{M}
\newcommand{\PMm}{M^{\prime}}
\newcommand{\inters}{\tau}						

\newcommand{\e}{\varepsilon}
\newcommand{\tilG}{\tilde{G}}
\newcommand{\tilg}{\tilde{g}}
\renewcommand{\aa}{a}
\newcommand{\bb}{b}
\newcommand{\inull}{\tilde{n}}

\newcommand{\Ebar}{\bar{E}}
\newcommand{\DZz}{\hk{}\hspace{-1pt}D\hspace{-1pt}Z\hspace{1pt}{}}
\newcommand{\DZ}{\hk{}\hspace{-1pt}D\hspace{-1pt}Z\hspace{1pt}{}}

\newcommand{\Cdot}{\,\raisebox{0.15ex}{\scalebox{.75}{$\bullet$}}\,}
\newcommand{\Ldot}{\,\raisebox{0.25ex}{\scalebox{.5}{$\blacktriangleleft$}}\,}
\newcommand{\Rdot}{\,\raisebox{0.25ex}{\scalebox{.5}{$\blacktriangleright$}}\,}

\newcommand{\N}{\mathbb{N}}
\newcommand{\R}{\mathbb{R}}
\newcommand{\Z}{\mathbb{Z}}
\newcommand{\Csp}{C}
\newcommand{\Cloc}{\hat{C}}

\newcommand{\Lsp}{L}
\newcommand{\Msp}{\mathcal{M}}
\newcommand{\filt}{\mathscr{F}}
\newcommand{\eps}{\varepsilon}

\newcommand{\norm}[1]{\Vert #1\Vert}

\newcommand{\NOrm}[1]{\Big\Vert #1\Big\Vert}
\newcommand{\normopm}[1]{\Vert #1\Vert_{\mathrm{op}/}}

\newcommand{\NOrmopm}[1]{\Big\Vert #1\Big\Vert_{\mathrm{op}/}}
\newcommand{\ip}[1]{\langle #1\rangle}
\newcommand{\Ip}[1]{\big\langle #1\big\rangle}

\newcommand{\cdott}{\hspace{1pt}\cdot\hspace{1pt}}

\newcommand{\E}{\mathbf{E}}
\newcommand{\EE}{\mathbb{E}}
\renewcommand{\P}{\mathbf{P}}
\newcommand{\PP}{\mathbb{P}}
\newcommand{\Var}{\mathbf{Var}}
\newcommand{\Varr}{\mathbb{V}\mathrm{ar}}
\newcommand{\ind}{1}
\newcommand{\img}{\mathbf{i}}
\newcommand{\sign}{\mathrm{sign}}					
\renewcommand{\d}{\mathrm{d}}		

\newcommand{\til}{\tilde}

\title[Log Log Fluctuations of the SHF]{Log Log Fluctuations of the Stochastic Heat Flow}
\author{Yu Gu}
\address[Yu Gu]{Department of Mathematics, University of Maryland, College Park, USA}

\author{Li-Cheng Tsai}
\address[Li-Cheng Tsai]{Department of Mathematics, University of Utah, USA}

\begin{document}
\begin{abstract}
We study the stochastic heat flow with constant initial data and analyze its spatial average on the scale of $\varepsilon\ll1$. We prove that the logarithm of the averaged process satisfies a pointwise central limit theorem: After being centered by $-\tfrac{1}{2}\log\log \varepsilon^{-1}$ and scaled down by $\sqrt{\log\log \varepsilon^{-1}}$, it converges in distribution to a standard Gaussian.
\end{abstract}

\maketitle

\section{Introduction}
\label{s.intro}

\subsection{Main results}
In this paper, we obtain the asymptotic pointwise fluctuations of the logarithm of the Stochastic Heat Flow (SHF) averaged over a small spatial region, for any fixed coupling constant $\theta\in\R$.
The SHF was first constructed in \cite{caravenna2023critical} as the finite dimensional distributional limit of discrete polymers in random environments.
In this paper we work with the axiomatic formulation of SHF in \cite{tsai2024stochastic}. 
Let $\Msp_+(\R^d)$ denote the space of positive locally finite measures on $\R^d$, equipped with the vague topology induced by continuous compactly supported functions on $\R^d$.
The SHF with coupling constant $\theta\in\R$ is the continuous $\Msp_+(\R^4)$-valued process $Z^{\theta}=Z^{\theta}_{s,t}$ in $s\leq t\in\R$ that is uniquely characterized by the axioms in \cite{tsai2024stochastic}, which we recall in Section~\ref{s.prelim.axioms}.
The existence of $Z^{\theta}$ under this formulation was also established in \cite{tsai2024stochastic}.
Let $\hk(s,x):=e^{-|x|^2/2s}/2\pi s$ denote the standard 2d heat kernel, fix any $\theta\in\R$, and consider
\begin{align}
	\label{e.h.Z}
	H_\e(y) :=\log Z_\e(y)\ ,
	&&
	Z_\e(y) := \int_{\R^4} Z^{\theta}_{0,1}(\d x, \d x')\,\hk(\e^2,y-x')\ ,
	\qquad
	x,x'\in\R^2\ .
\end{align}
\begin{thm}\label{t.main}
 There exists a sequence of deterministic constants $\alpha_\e\to 0$ such that
\begin{align}\label{e.mainresult}
	\frac{1}{\sqrt{\log\log\e^{-1}}}\Big(
		H_\e(0) + \frac{1+\alpha_\e}{2} \log\log\e^{-1}
	\Big)
	\Longrightarrow
	\text{standard Gaussian,}
    \quad
    \text{ as $\e\to 0$\ .}
\end{align}
\end{thm}

Utilizing the spatial translation invariance of $H_\e$, one immediately obtains the following corollary that says the ``typical'' value of $Z_\e$ is of order $1/\sqrt{\log\e^{-1}}$.
\begin{cor}\label{c.main}
 There exists $\beta_\eps<\beta_\eps'$ such that $\beta_\eps,\beta_\eps'\to\tfrac12$ as $\eps\to0$ and for
\begin{align}
D_{\e} := \Big\{ y\in\R^2\, \Big| \, 
        \frac{1}{(\log\e^{-1})^{\beta_\eps'}} 
        \leq Z_\e(y) 
        \leq \frac{1}{(\log\e^{-1})^{\beta_\eps}} 
    \Big\}
\end{align} 
and all bounded $D\subset\R^2$,
\begin{align}
	\E |D\cap D_{\e}^{\compl}|   \xlongrightarrow{\e} 0\ .
\end{align}
\end{cor}

\subsection{Challenge caused by intermittency} 
%
The main challenge in proving Theorem~\ref{t.main} is that it requires obtaining the typical behaviors of $Z_\eps(0)$ in the presence of intermittency.
It is well-known that $\E Z_\eps(0)=1$ and $\E Z_\eps(0)^2\sim \log \eps^{-1}$. 
For an analog of $Z_\e(0)$, it was further shown in \cite{liu2024moments} that $\E (Z_\eps(0))^n \sim (\log \eps^{-1})^{\binom{n}{2}}$, for every fixed $n\in\N$.
These moment asymptotics suggest that $Z_\eps(0)$ is typically small, while the moments are actually dominated by rare and anomalously large values of $Z_\eps(0)$, which is known as intermittency.
Positive-integer moments have been and remain a major tool for understanding the SHF and related models. 
In the presence of intermittency, however, it becomes significantly more challenging to extract typical behaviors of $Z_\eps(0)$ from these moments.

To put things into perspective, we note that a series of works have established the convergence of the analogs of $Z_\e(0)$ to lognormal distributions, see \cite{caravenna17,dunlap2022forward,caravenna2022gaussian,caravenna2025singularity} for results on the directed polymers, the mollified stochastic heat equation, and the averaged SHF in certain weak disorder regimes. In all those works, the analogs of $Z_\e(0)$ converge in law to lognormal $\exp(\sigma N(0,1)-\sigma^2/2)$ with some $\sigma\in(0,\infty)$. In particular, the limit of $Z_\eps(0)$ is of order $1$, and the second moments remain bounded as $\eps\to0$.
On the other hand, the current paper deals with the genuinely critical regime and proves 
that,  in the sense of \eqref{e.mainresult},
\begin{align}
	\label{e.biglognormal}
	Z_\e(0) \text{ approximates } \exp\Big( \sigma_{\e} N(0,1) - \frac{1}{2}\sigma^2_{\e} \Big) \text{ in law,}
	&&
	\sigma^2_\e := \log\log\e^{-1}\to\infty\ . 
\end{align}
Here $Z_\e(0)$ decays to $0$ at the rate of $1/\sqrt{\log\e^{-1}}$ and the second moment of $Z_\e(0)$ tends to $\infty$ at the rate of $\log\e^{-1}$.

Intermittency has significant implications in the proof.
To illustrate this point, 
for some $\bb\ll1$, consider the exponentially scaled times $\{t_i:=\e^{2}\bb^{-2i}\}_{i=0,\ldots,N}$ which interpolate between $\eps^2$ and $1$. 
A key property we establish in the paper is that $Z_\e(0)$ approximately decouples under the exponential timescale. While the ``decoupling'' was already observed a few times in previous works, in those weak-disorder regimes that are non-intermittent, the approximate decoupling holds because the second moment of (the analogues of)
\begin{align}
	\label{e.Z-Zdecoupled}
	Z_\e(0) - \prod_{i=1}^{N} \int_{\R^4} Z^{\theta}_{t_{i-1},t_{i}}(\d x, \d x') \hk(t_{i-1}, x) 
\end{align}
converges to $0$, so that $Z_\eps(0)$ can be approximated by the product of independent random variables. 
See e.g.\ \cite[Theorem~2.2]{cosco2025central} where this was explicitly shown.
In the intermittent regime, however, the second moment of \eqref{e.Z-Zdecoupled} \emph{diverges to $\infty$} while both terms there \emph{decay to $0$}.
Even though the approximate decoupling still holds, it holds for a very different reason.
In fact, since both terms in \eqref{e.Z-Zdecoupled} decay to 0, we need to consider their \emph{ratio}
\begin{align*}
	\frac{Z_\e(0)}{\prod_{i=1}^{N} \int_{\R^4} Z^{\theta}_{t_{i-1},t_{i}}(\d x, \d x') \hk(t_{i-1}, x)}
\end{align*}
(more precisely a truncated version of it; see \eqref{e.step1.1}) and bound it away from $0$ and $\infty$.

What makes the intermittent regime particularly challenging is that we need to work with a divergent number of intervals that ultimately leads to the divergence of the mean and variance in \eqref{e.mainresult}. 
In previous works, the analogue of $\bb$ was defined as $\e^{1/N}$ so the exponential time scale reads as $\eps^q$ with the variable  $q=\tfrac{2N-2i}{N}\in [0,2)$, and the general strategy there was to send $\e\to 0$ first while keeping $N$ fixed then send $N\to\infty$ later.
In the current work, we (have to) send $\e\to 0$ first and $\bb\to 0$ later, which requires working with $N=\lfloor \log\e^{-1}/\log\bb^{-1}\rfloor$ that diverges in $\e$. From a probabilistic point of view, the decoupling can be understood as follows. Take any $s<t$, the semigroup property of SHF enables us to write $Z^\theta_{0,t}=Z^\theta_{0,s} \Cdot Z^\theta_{s,t}$; see \eqref{e.Cdot}. One can view $Z^\theta_{0,s}$ as the point-to-point partition function of a directed polymer on $[0,s]$ and the decoupling we are looking for essentially amounts to approximating the point-to-point partition function by the point-to-line partition function multiplied by the heat kernel $p(s,\cdot)$: \[
Z^\theta_{0,s} (\d x, \d y)\approx \int_{x'\in\R^2} Z_{0,s}^\theta(\d x, \d x') \ p(s,x-y)\,\d y\ .
\]Although the error in such approximations can be controlled for each $[s,t]=[t_{i-1},t_i]$, one of the main technical challenges in proving Theorem~\ref{t.main} is to ensure that such errors do not pile up as the number of intervals goes to infinity.


\subsection{Toward KPZ}
Another motivation of our study is to understand the scaling limit of the height functions of related random growth models or the free energy of directed polymers.
Indeed, at an informal level, the logarithm of the Stochastic Heat Equation (SHE) gives the Kardar--Parisi--Zhang (KPZ) equation that describes random growth.
Since the SHF can be obtained as the scaling limit of the 2d mollified SHE, it is interesting to consider taking the logarithm of the SHF and to study the resulting process.
However, since the SHF is genuinely measure-valued, as shown in Theorem 1.1 of \cite{caravenna2025singularity} and quantified in Corollary~\ref{c.main} of the current paper, even just taking the logarithm of the SHF is subtle.
Theorem~\ref{t.main} of the current paper can be viewed as a result toward the 2d KPZ equation beyond the subcritical regime, done through taking the logarithm of spatially averaged SHF and studying the pointwise fluctuations.
An interesting next question concerns the fluctuations as a field, namely a random Schwartz distribution.
In the subcritical regime, the fluctuations as a field of the 2d KPZ equation have been shown to converge to the Edwards--Wilkinson Gaussian limit; see \cite{chatterjee18,caravenna18a,gu2020,nakajima2023fluctuations}.
In this regime, understanding the pointwise fluctuations plays a crucial role in obtaining the field convergence.
We hope that the current paper could similarly offer an entry point for understanding the 2d KPZ equation beyond the subcritical regime.

\subsection{Some related literature}
\label{s.intro.liter}

In the subcritical/weak disorder regime, the lognormal fluctuations of the partition function of directed polymers in $1+2$ random environments and the solutions to the mollified 2d stochastic heat equation were first derived in \cite{caravenna17}, and it was further shown there that, at and beyond the critical regime, the same quantity goes to zero in probability.  The work of  \cite{dunlap2022forward,dunlap2023forward2} generalized the results in the subcritical setting to the model of semilinear stochastic heat equations and derived distributional limits in the form of a forward-backward SDE. Later a simplified proof of lognormality was given in \cite{caravenna2022gaussian} from the perspective of chaos expansion. Most relevant to our work is \cite{cosco2025central}, where a short and simple proof of the lognormal fluctuations was given by establishing the decoupling in \eqref{e.Z-Zdecoupled}. 

A quasi-critical regime was introduced in \cite{caravenna2023quasi}, which in some sense interpolates between the subcritical and critical regimes. For directed polymers, lognormal fluctuations were established in  \cite{caravenna2025singularity}, in the quasi-critical and critical regimes, where the inverse temperature (resp.\ the coupling constant $\theta$) was tuned in a way so that the partition function stays bounded in $L^2$. 
It was further shown in \cite{caravenna2025singularity} that SHF is singular with respect to the Lebesgue measure and the convergence of $Z_\eps(y)\to0$ for Lebesgue a.e. $y\in \R^2$.

The double $\log$ appearing in the statement of Theorem~\ref{t.main} comes from a combination of the logarithmic correlation of SHF together with the nonlinear transformation $H_\eps=\log Z_\eps$. Incidentally, a similar phenomenon has been observed in the study of intermittent behaviors of a passive tracer in the curl of 2d Gaussian free field \cite{morfe2025critical}.



\subsection{Outline of proof}
\label{s.intro.proof}
First, let us state the result that we actually prove as Theorem~\ref{t.main.}, which is slightly stronger than Theorem~\ref{t.main}.
To streamline notation, we use the time-reversal invariance of the SHF to move the heat kernel in \eqref{e.h.Z} to the starting time, and further change the starting time from $0$ to $\e^2$.
The latter causes no loss of generality because $[\e^2,1]$ can be scaled back to $[0,1]$ by the invariance of the SHF (see \cite[Corollary~1.6]{tsai2024stochastic}) at the expense of changing $\theta$ slightly, and we allow $\theta$ to vary in any fixed bounded interval in Theorem~\ref{t.main.}.
\begin{thm}\label{t.main.}
Allowing $\theta$ to depend on $\e$ with $|\theta|=|\theta_\e|\leq c_0<\infty$, we have
\begin{align}
	\label{e.t.main.}
	\frac{1}{\sqrt{\log\log\e^{-1}}}\Big(
		\log \int_{\R^4} Z^{\theta}_{\e^2,1}(\d x, \d x')\,\hk(\e^2,x) + \frac{1+\alpha_\eps}{2} \log\log\e^{-1}
	\Big)
	\Longrightarrow
	\text{standard Gaussian}.
\end{align}
\end{thm}

We will work with the exponentially scaled times
\begin{align}
	\label{e.ti}
	&t_i = t_{i,N,\bb} := \e^{2}\bb^{-2i}\ ,
	\
	i=0,\ldots, N-1\ ,
	&&
	N = N_{\e,\bb} := \lfloor \log\e^{-1}/\log \bb^{-1}\rfloor\ ,
\end{align}
The last few factors in \eqref{e.Z-Zdecoupled} behave poorly, so those time intervals need separate treatment.
To this end, we take an $\ell$, consider $[t_{i-1},t_{i}]$ for $i=1,\ldots,N-\ell$ and separately $[t_{N-\ell},1]$.
We will send $\e\to 0$ (whence $N\to\infty$) first, and $(\ell,\bb)\to(\infty,0)$ later.
Accordingly, our proof of Theorem~\ref{t.main} consists of two steps.
The first step deals with the intervals $[t_{i-1},t_{i}]$ for $i\leq N-\ell$, and the second step deals with $[t_{N-\ell},1]$.

In the first step of the proof, we show that
\begin{align}
	\label{e.step1.1}
	\frac{1}{2} \leq
	\frac{\int_{\R^4} Z^{\theta}_{\e^2,t_{N-\ell}}(\d x, \d x')\,\hk(\e^2,x)}{\prod_{i=1}^{N-\ell} \int_{\R^4} Z^{\theta}_{t_{i-1},t_{i}}(\d x, \d x') \hk(t_{i-1}, x)}
	\leq \frac{3}{2}\ 
\end{align}
with high probability, as $\ell\to\infty$ uniformly in $\e\leq 1/2$ and $\bb\leq 1/c$, and that 
\begin{align}
	\label{e.step1.2}
	\frac{1}{\sqrt{\log\log\e^{-1}}} \Big( 
		\log \text{ denominator of \eqref{e.step1.1} }
		+ \frac{1+o(1)}{2}\log\log\e^{-1}
	\Big)
	\Longrightarrow \text{ standard Gaussian.}
\end{align}
Given that the denominator of \eqref{e.step1.1} is a product of independent variables, the statement \eqref{e.step1.2} can be proven by the classical central limit theorem with the aid of suitable moment estimates.
We do this in Lemma~\ref{l.clt} and Appendix~\ref{s.a.Wmom}.
On the other hand, proving \eqref{e.step1.1} makes up the bulk of the entire paper.
To prove it, we start by developing an exact expansion of the fraction in \eqref{e.step1.1} in Proposition~\ref{p.expansion} and spend Sections~\ref{s.boundingerror}--\ref{s.decoupling} bounding the error terms in the expansion.

In the second step of the proof, we show that for $r'_{\ell,\bb},r_{\ell,\bb}\to\infty$ fast enough as $(\ell,\bb)\to(\infty, 0)$,
\begin{align}
	\label{e.step2}
	\P\Big[
		e^{-r_{\ell,\bb}} \leq
		\frac{\int_{\R^4} Z^{\theta}_{\e^2,t_{N-\ell}}(\d x, \d x')\,\hk(\e^2,x)}{\int_{\R^4} Z^{\theta}_{\e^2,1}(\d x, \d x')\,\hk(\e^2,x)}
		\leq e^{r'_{\ell,\bb}}
	\Big] 
	\longrightarrow 1
\end{align}
uniformly in $\e\leq 1/2$.
The key step is to derive a lower-tail bound on the fraction in \eqref{e.step2}.
To this end, we invoke the conditional Gaussian Multiplicative Chaos (GMC) structure from \cite{clark2025conditional} and apply the argument in \cite{morenoflores2014positivity}, which was based on (one of) Talagrand's Gaussian concentration inequalities, to the conditional GMC.
Doing this requires some bounds that we developed in Sections~\ref{s.boundingerror}--\ref{s.decoupling}.
We note that a lower-tail bound of (a variant of) $Z_\e(0)$ has been obtained in \cite{nakashima2025upper}.

The precise versions of \eqref{e.step1.1}--\eqref{e.step2} will be given in Lemma~\ref{l.clt}, Corollary~\ref{c.clt}, and \eqref{e.upbd}--\eqref{e.lwbd} respectively.
Indeed, combining them gives Theorem~\ref{t.main.}.

\subsection*{Outline of paper}
In Section~\ref{s.prelim}, we recall the axiomatic formulation of \cite{tsai2024stochastic}, prepare a few tools, and develop an expansion of the quantity of interest $\int_{\R^4} Z^{\theta}_{\e^2,1}(\d x, \d x')\hk(\e^2,x) $. 
Section~\ref{s.boundingerror} is devoted to bounding the second moment of an expression involved in the expansion.
Based on the bound, in Section~\ref{s.decoupling} we bound all error terms together in the expansion.
This completes the first step stated in Section~\ref{s.intro.proof}.
The second step is carried out in Section~\ref{s.last.piece} by establishing certain comparison bounds.
Some technical steps are placed in the appendices to streamline the presentation.

\subsection*{Acknowledgment}
We thank Cl\'{e}ment Cosco for pointing us to the paper \cite{cosco2025central} and the role of the semigroup property in that paper. 
The research of Y.G.\ is partially supported by the National Science Foundation under grant no.\ DMS-2203014.
The research of LCT is partially supported by the NSF through DMS-2243112 and the Alfred P Sloan Foundation
through the Sloan Research Fellowship FG-2022-19308.

\section{Axiomatic formulation, tools, decoupling expansion}
\label{s.prelim}

\subsection{Axiomatic formulation of SHF}
\label{s.prelim.axioms}
We recall the axiomatic formulation of the SHF from \cite{tsai2024stochastic} that we will work with.
Deferring the definition of the product $\Cdot$ and the function $\sg^{n,\theta}$, we recall that by \cite[Theorem~1.3]{tsai2024stochastic}, the SHF with coupling constant $\theta\in\R$ is uniquely characterized by the following Axioms.
\begin{enumerate}
\item \label{d.shf.}
$Z^{\theta}=Z^{\theta}_{s,t}$ is an $\Msp_+(\R^4)$-valued, continuous process in $s\leq t\in\R$.
\item \label{d.shf.ck}
For all $s< t< u$, $Z^{\theta}_{s,t}\Cdot Z^{\theta}_{t,u}=Z^{\theta}_{s,u}$.
\item \label{d.shf.inde}
For all $s< t< u$, $Z^{\theta}_{s,t}$ and $Z^{\theta}_{t,u}$ are independent.
\item \label{d.shf.mome}
For $n=1,\ldots, 4$, $s<t$, and $x=(x_1,\ldots,x_n),x'=(x'_1,\ldots,x'_n)\in\R^{2n}$,
\begin{align}
    \E\, \bigotimes_{i=1}^n Z^{\theta}_{s,t}(\d x_i,\d x'_i)
    =
    \d x\, \d x' \, \sg^{n,\theta}(t-s,x,x')\ .
\end{align}
\end{enumerate}
The existence of $Z^{\theta}$ under this formulation was established in \cite[Proposition~1.5]{tsai2024stochastic}.

The product $\Cdot$ was introduced in \cite{clark2024continuum} and reads as follows
\begin{align}
	\label{e.Cdot}
	Z^{\theta}_{s,t}\Cdot Z^{\theta}_{t,u}
	&:=
	\lim_{\delta\to 0}
	Z^{\theta}_{s,t}\Cdot_{\delta} Z^{\theta}_{t,u}\ ,
\\
	\label{e.Cdot.}
	\big( Z^{\theta}_{s,t}\Cdot_{\delta} Z^{\theta}_{t,u} \big)(\d x, \d x')
	&:=
	\int_{\R^{4}} Z^{\theta}_{s,t}(\d x, \d y)\, \hk(\delta^2,y-y')\, Z^{\theta}_{t,u}(\d y', \d x')\ .
\end{align}
The work \cite{clark2024continuum} showed that \eqref{e.Cdot.} converges vaguely a.s.\ and in $\Lsp^2(\P)$, (more precisely a slight variant of \eqref{e.Cdot.} where the two $t$s on the right hand side are respectively replaced by $t_\delta,t'_{\delta}$ with $t_{\delta}\leq t\leq t'_{\delta}$ and $t'_{\delta}-t_{\delta}=\delta$).
Later \cite{tsai2024stochastic} showed that \eqref{e.Cdot.} converges vaguely in $\Lsp^{n}(\P)$ for all $n<\infty$ and the heat kernel $\hk$ in \eqref{e.Cdot.} can be replaced by a general class of functions.

The function $\sg^{n,\theta}(t,x,x')$ is the integral kernel of the 2d delta-Bose semigroup $\sg^{n,\theta}(t)$, a strongly continuous semigroup on $\Lsp^2(\R^{2n})$.
Its explicit formula for $n\in\N$ was obtained in \cite[Theorem~1.1]{gu2021moments} based on \cite{rajeev99,dimock04}.
For $n=1$, we simply have $\sg^{1,\theta}(t)=\hk(t)$.
The $n=2$ formula dates back to \cite{albeverio1995fundamental,bertini98} and reads as follows. 
Letting
\begin{align}
	\label{e.jfn}
	\jfn^{\theta}(t)
	&:=
	\int_{0}^\infty \d a \, \frac{t^{a-1}e^{\theta a}}{\Gamma(a)}\ ,
\\
	\begin{split}
	\label{e.sgg2}	
	\sgg^{2,\theta}(t,x,x')
 	&:=
 	\int_{0<s<s'<t} \d s \d s' \int_{\R^4} \d y \d y'\,
 	\prod_{j=1,2} \hk(s,x_j-y) 
 \\
 	&\qquad\cdot  
 	4\pi \jfn^{\theta}(s'-s) \hk(\tfrac{s'-s}{2},y-y') \prod_{j=1,2} \hk(t-s',y'-x'_j)\ ,
	\end{split}
\end{align}
we have that
\begin{align}
	\label{e.sg2}
    \sg^{2,\theta}(t)
    =
    \hk(t)^{\otimes 2} + \sgg^{2,\theta}(t)\ .
\end{align}
For the bulk of this paper, we will only use the formulas for $n=1,2$.
Only in Appendix~\ref{s.a.Wmom} do we use the general $n$ formula, and we will recall it there.

\subsection{Notation and tools}
\label{s.prelim.tools}
In this paper, we adopt notation that we hope is suggestive for writing the action of measure-valued processes on test functions such as
\begin{align}
	\int_{\R^4} Z^{\theta}_{s,t}(\d x, \d x')\, f(x) g(x')
	:=
	f \Ldot Z^{\theta}_{s,t} \Rdot g\ ,
\end{align}
and similarly for products of them such as $ (f\Ldot Z^{\theta}_{s,t})^{\otimes 2} \Rdot h$, where $h=h(x_1,x_2)$, $x_1,x_2\in\R^2$.
The left/right solid arrow points to the test function at the earlier/later time.
Together with the semigroup property in Axiom~\ref{d.shf.ck}, we have identities such as
\begin{align}
	\label{e.dot.ex}
	f \Ldot Z^{\theta}_{s,u} \Rdot g
	=
	f \Ldot Z^{\theta}_{s,t} \Cdot Z^{\theta}_{t,u} \Rdot g\ .
\end{align}

We will often decompose $Z^{\theta}$ into its mean and mean-zero part.
Since $\sg^{1,\theta}(t)=\hk(t)$, letting $\heatsg_{s,t}=\heatsg_{s,t}(\d x, \d x'):= \d x \d x'\, \hk(t-s,x-x')$, we have that $\E Z^{\theta}=\heatsg$.
Accordingly, we write $W^{\theta} := Z^{\theta} - \heatsg$ for the mean-zero part of $Z^{\theta}$.
By definition, $\E W^{\theta}=0$, and it is readily checked from \eqref{e.sgg2} that $\E W^{\theta}_{s,t}{}^{\otimes 2} = \d x\, \d x'\, \sgg^{2,\theta}(t-s,x,x')$.
As said, we will mostly work with the $n=1,2$ moments, but we also mention that the $n$th moment of $W^{\theta}$  permits an explicit formula
\begin{align}
	\label{e.Wmomformula}
	\E\, W^{\theta}_{s,t}{}^{\otimes n} (\d x, \d x')
	:=
	\E\, \bigotimes_{i=1}^n W^{\theta}_{s,t}(\d x_i,\d x'_i)
	=
	\d x\, \d x'\, \sgg^{n,\theta}(t-s,x,x')\ ,
\end{align}
where the expression of $\sgg^{n,\theta}$ is recalled in \eqref{e.sgg}.

Next, we note that upon taking expectation, the products $\Cdot,\Ldot,\Rdot$ translate into operator product and inner product.
Write $\ip{\cdott,\cdott}$ for the inner product on $\Lsp^2(\R^d, \d x)$.
For $f,g\in \Lsp^2(\R^{2n})$, 
\begin{align}
	\label{e.dot.underE}
	\E \, f \Ldot (Z^{\theta}_{s,t} \Cdot W^{\theta}_{t, u})^{\otimes n} \Rdot g
	=
	\Ip{ f\, , \sg^{n,\theta}(t-s)\, \sgg^{n,\theta}(u-t)\, g }\ ,
\end{align}
as can be verified from Axiom~\ref{d.shf.mome}, \eqref{e.Cdot}, and \eqref{e.Wmomformula}, and similar relations hold for products with more terms.

The function $\jfn^{\theta}(t)$ in \eqref{e.jfn} approximates $t^{-1}(\log t+\theta)^{-2}$ for small $t$, more precisely
\begin{align}
	\label{e.jfn.estimate}
	\frac{(1-\epsilon_-(c_0,t))}{t\,((\log t+ \theta)_-^2+1)}
	\leq
	\jfn^{\theta}(t) 
	\leq 
	\frac{(1+\epsilon_+(c_0,t))}{t\,((\log t+ \theta)_-^2+1)}\ ,
	\qquad
	\theta\leq c_0\, ,\ t\in(0,1]\ ,
\end{align}
where $\epsilon_\pm(c_0,\cdott)$ are bounded and satisfy $\epsilon_{\pm}(c_0,t)\to 0$ as $t\to 0$.
In particular,
\begin{align}
	\label{e.jfn.intbd}
	\int_0^t \d s\, \jfn^{\theta}(s)
	\leq
	\frac{c(c_0)}{(\log t+ \theta)_-+1}\ ,
	\qquad
	t\in(0,1]\ , \ \theta\leq c_0\ .
\end{align}
The estimate \eqref{e.jfn.estimate} can be obtained for example from the following inverse Laplace formula.
Note that the Laplace transform of $\jfn^{\theta}$ gives $\int_{0}^\infty \d t\, e^{zt}\jfn^{\theta}(t)=1/(\log(-z)-\theta)$ for $-z>e^{\theta}$, so inverting the Laplace transform through the Mellin-like inversion formula gives
\begin{align}
	\jfn^{\theta}(t)
	=
	\int_{\subset} \frac{\d z}{2\pi\img}\frac{e^{-zt}}{\log(- z) - \theta}\ ,
\end{align}
where $\subset$ denotes the counterclockwise contour $\{[-e^{\theta},\infty)+\img 0\}\cup \{[-e^{\theta},\infty)-\img 0\}$ that wraps $[-e^{\theta},\infty)$ in the complex plane from just above and below the real axis.

Next, let us state a couple useful bounds.
Let $\filt_{[s,t]}$ be the sigma algebra generated by $Z^{\theta}_{s',t'}$ for $s\leq s'\leq t'\leq t$, recall $t_i$ from \eqref{e.ti}, let
\begin{align}
	Z_i &= Z_i^{\theta} := Z^{\theta}_{t_{i-1},t_{i}}\ ,
	&
	W_i &= W_i^{\theta} := W^{\theta}_{t_{i-1},t_{i}}\ ,
	&
	\hk_i &:= \hk(t_i,\cdott)\ ,
\end{align}
and write $\hk_{i,r,y}:=\hk(t_ir^2,\cdott-y)$ to simplify notation.
The following bounds are not hard to verify from \eqref{e.ti} and \eqref{e.sgg2} with the aid of \eqref{e.jfn.estimate}--\eqref{e.jfn.intbd}:
\begin{align}
	\label{e.W.2ndmom}
	&\lim_{\bb\to 0} \sup_{\e\leq 1/2,i \in[1,N-1]}
		\big| (N-i) \E (\hk_{i-1} \Ldot W^{\theta}_i \Rdot 1)^2 - 1 \big|
	= 0\ ,
\\	
	\label{e.Z.hkk}
	&t_{i}r^2 \int_{\R^2} \d y \, \E (\hk_{i-1} \Ldot Z^{\theta}_{i} \Rdot \hk_{i,r,y})^2
	\leq
	c(c_0) r^2\log\frac{1+r^2}{r^2}\ ,
\end{align}
where \eqref{e.Z.hkk} holds for all $i\in[1,N-1]$, $\bb\leq 1/2$, $|\theta|\leq c_0$, and $r\in(0,\infty)$.

To alleviate heavy notation, we will often drop the dependence on $\e,\bb,\theta$ as we have been doing by writing $N_{\e,\bb}=N$, $t_{i,\e,\bb}=t_i$, and $Z^{\theta}_i=Z_i$.
Unless otherwise noted, all bounds in this paper hold for $|\theta|\leq c_0$ for any fixed $c_0<\infty$ and other constants $c$ in this paper may depend on $c_0$ \emph{without explicit specification}.
We write $c=c(a_1,a_2,\ldots)$ for a general deterministic constant in $[2,\infty)$ that may vary from place to place but depend only on the designated variables, except that we may omit the dependence on $c_0$.

\subsection{Decoupling expansion}
\label{s.prelim.decoupling}
Let us develop an exact expansion of the fraction in \eqref{e.step1.1}.
Set
\begin{align}
	\label{e.Di}
	D_i &= D_i^{\theta} := Z_{i}^{\theta} - ( Z_{i}^{\theta} \Rdot 1 )\, \hk_{i}\ ,
\\
	\label{e.scrI}
	\begin{split}	
	\scrI &= \scrI(N,\ell)
	:= 
	\big\{ 
		\calI = \{ I_1,\ldots,I_{|\calI|} \} \, \big| \, 
		|\calI| \geq 1, 
\\
		&I_1,\ldots, I_{|\calI|} \text{ disjoint intervals of integers in } [1,N-\ell], 
		\text{ each with } |I_i| \geq 2
	\big\}\ ,
	\end{split}		
\end{align}
and write any nonempty interval of integers $I$ as $[i_I,j_I]\cap\Z$ with $i_I\leq j_I\in\Z$. One may view $D_i$ as the error in approximating the point-to-point partition function $Z_i^\theta$ by the point-to-line partition function $ Z_{i}^{\theta} \Rdot 1$ multiplied by the heat kernel $ \hk_{i}$. 
\begin{prop}\label{p.expansion}
We have the expansion:
\begin{align}
	\label{e.decoupling}
	\frac{ \hk(\e^{2}) \Ldot Z_{\e^2,t_{N-\ell}} \Rdot 1}{ \prod_{i=1}^{N-\ell} \hk_{i-1} \Ldot Z_{i} \Rdot 1 }
	&=
		1
		+
		\sum_{\calI\in\scrI}
		\prod_{I\in\calI} 
		\frac{
			\hk_{i_{I}-1} \Ldot D_{i_{I}} \Cdot \cdots \Cdot D_{j_{I}-1} \Cdot Z_{j_I} \Rdot 1 
		}{
			(\hk_{i_{I}-1} \Ldot Z_{i_{I}} \Rdot 1) \cdots (\hk_{j_{I}-1} \Ldot Z_{j_{I}} \Rdot 1)
		}\ .
\end{align}
\end{prop}
\begin{proof}
To begin, write
\begin{align}
	\label{e.decoupling.1}
	\hk(\e^{2}) \Ldot Z_{\e^2,t_{N-\ell}} \Rdot 1
	=
	\big( \hk_{0} \Ldot Z_{1} \Rdot 1 \big)
	\ 
	\hk_{1} \Ldot Z_{t_1,t_{N-\ell}} \Rdot 1 
	+
	\hk_{0} \Ldot D_{1} \Cdot Z_{t_1,t_{N-\ell}} \Rdot 1\ .
\end{align}
On the right hand side, the first term decouples the actions on $[\e^2,t_1]$ and $[t_1,t_{N-\ell}]$, while the second term is the error term.
Continue the procedure by writing
\begin{align}
	\hk_{1} \Ldot Z_{t_1,t_{N-\ell}} \Rdot 1
	&=
	\big( \hk_{1} \Ldot Z_{t_1,t_{2}} \Rdot 1 \big)\ 
	\hk_{2} \Ldot Z_{t_2,t_{N-\ell}} \Rdot 1
	+
	\hk_{1} \Ldot D_{2} \Cdot Z_{t_2,t_{N-\ell}} \Rdot 1\ ,
\\
	\hk_{0} \Ldot D_{1} \Cdot Z_{t_1,t_{N-\ell}} \Rdot 1
	&=
	\big( \hk_{0} \Ldot D_{1} \Cdot Z_{t_1,t_2} \Rdot 1	\big) \ \hk_2 \Ldot Z_{t_2,t_{N-\ell}} \Rdot 1
	+
	\hk_{0} \Ldot D_{1} \Cdot D_{2} \Cdot Z_{t_2,t_{N-\ell}} \Rdot 1\ ,
\end{align}
and continue similarly.
Doing so gives
\begin{subequations}
\label{e.decoupling.}
\begin{align}
	\label{e.decoupling.pre.1}
	\hk(\e^{2}) &\Ldot Z_{\e^2,t_{N-\ell}} \Rdot 1
	=
	\prod_{i=1}^{N-\ell} \hk_{i-1} \Ldot Z_{i} \Rdot 1	
\\
	&+
	\sum_{\calI\in\scrI}
	\prod_{I\in\calI} \hk_{i_I-1} \Ldot D_{i_I} \Cdot \cdots \Cdot D_{j_I-1} \Cdot Z_{j_I} \Rdot 1	\cdot
	\prod_{i\notin \cup\calI} \hk_{i-1} \Ldot Z_{i} \Rdot 1\ .
\end{align}
\end{subequations}
Dividing both sides by the first product gives the desired result.
\end{proof}

To control the denominator on the right hand side of \eqref{e.decoupling}, we introduce the event
\begin{align}
	\label{e.Wevent}
	\Wevent = \Wevent(N_{\e,\bb},\ell) := \cap_{i=1}^{N-\ell} \Wevent_{i}\ ,
	\qquad
	\Wevent_i
	:=
	\big\{ |\hk_{i-1}\Ldot W_i\Rdot 1| \leq 1/2 \big\}\ .
\end{align}
Note that $\hk_{i-1}\Ldot Z_i\Rdot 1= 1 + \hk_{i-1}\Ldot W_i\Rdot 1$, so on the event $\Wevent$ we have $\hk_{i-1}\Ldot Z_i\Rdot 1 \geq 1/2$.
This event holds with high probability as $\ell\to\infty$.
To see why, use \eqref{e.Wi.mom} for $n=4$ to get $\P[\Wevent_i^\compl]\leq c(c_0,n)(N-i)^{-2}$ for all $i\leq N-c$, and combine them through the union bound to get
\begin{align}
	\label{e.Wevent.prob}
	\sup_{\e,\bb\in(0,1/2]} \P[\Wevent^\compl]
	\longrightarrow 0 \ ,
	\quad
	\text{as } \ell\to\infty\ ,
	\qquad
	\text{where }\Wevent = \Wevent(N_{\e,\bb},\ell)\ .
\end{align}

The denominator on the left hand side of \eqref{e.decoupling} is a product of independent variables.
With the aid of \eqref{e.W.2ndmom} and a higher moment bound, which we derive in \eqref{e.Wi.mom}, it is not hard to apply the classical central limit theorem to obtain the following convergence.
We include a proof of it at the end of Appendix~\ref{s.a.Wmom}.

\begin{lem}\label{l.clt}
There exist  $f(\eps,\ell,\bb)$ such that $\lim_{(\ell,\bb)\to (\infty,0)}\limsup_{\eps\to0} f= 0$  and that the random variable
\begin{align}
	H'_{\e,\ell,\bb}
	:=
	\frac{1 }{\sqrt{\log\log\e^{-1}}}
	\Big(
		\log\prod_{i=1}^{N-\ell} \hk_{i-1}\Ldot Z_{i} \Rdot 1
		+
		\frac{1+f(\eps,\ell,\bb)}{2}\log\log\e^{-1}
	\Big)
\end{align}
converges to Gaussian in the following sense
\begin{align}
    \label{e.clt}
	\lim_{\ell,\bb\to(\infty, 0)}
	\limsup_{\e\to 0}
	\Big| \P\big[ H'_{\e,\ell,\bb} \in [\alpha,\beta] \big] - \int_{\alpha}^{\beta} \frac{\d x}{\sqrt{2\pi}} \, e^{-x^2/2} \Big|
	=
	0\ ,
    \qquad
    \alpha<\beta\in\R\ .
\end{align}
\end{lem}

\section{Bounding the errors}
\label{s.boundingerror}

In this section, we bound the second moment of the numerator in the fraction in \eqref{e.decoupling}.
Let
\begin{align}
	\label{e.DZ}
	\DZ_{I} := \hk_{m-1} \Ldot D_{m} \Cdot \cdots \Cdot D_{n-1} \Cdot Z_{n} \Rdot 1\ ,
	\quad
	I = [m, n]\ 
\end{align}
be the quantity of interest.
Throughout this section, we write intervals of indices as $[i,j]:=[i,j]\cap\Z$, implicitly assuming the \emph{restriction onto integers}.
We will also write $I=[i_I,j_I]=[m,n]$ interchangeably for convenience of notation.
To prepare for the analysis in Section~\ref{s.last.piece}, we actually consider a slightly different expression 
\begin{align}
	\label{e.DZz}
	\int_{\R^2} \d y \, \E\,\DZz_{I,r}(y)^2\ ,
	&&
	\DZz_{I,r}(y) := \hk_{m-1} \Ldot D_{m} \Cdot \cdots \Cdot D_{n-1} \Cdot Z_{n} \Rdot \hk(t_nr^2,\cdott-y)\ .
\end{align}
As will be seen in the proof of Corollary~\ref{c.DZ.bd}, one can easily deduce a bound on $\E\DZ_{I}^2$ from a suitable bound on the integral in \eqref{e.DZz}.

To bound the integral in \eqref{e.DZz}, we begin by deriving an expression of it.
\begin{lem}\label{l.DZz.int}
For $I=[m,n]$ with $m<n$ and the $B_{\omega}$ defined after this lemma,
\begin{align}
	\label{e.Dzz=sumB}
	\int_{\R^2} \d y \, \E\ \DZz_{I,r}(y)^2
	=
	\sum_{\omega\subset(m,n]} B_\omega\ .
\end{align}
\end{lem}
\noindent{}%
The definition of $B_\omega$ will involve the index set $[m,n]_{\omega} := [m,n]\setminus\omega$, where we have and will continue to write $J_{\omega}:=J\setminus\omega$.
Let $i_-=i_-(\omega):=\max [m,i)_{\omega}$ be the previous index of $i$ in the index set.
The definition will also involve sums over $\vecbeta_i\in\{+,-\}^2$ for $i\in[m,n]_{\omega}$.
Write $\vecbeta_{A}:=(\vecbeta_{i})_{i\in A}$ to simplify notation.
For $\vecbeta=(\beta_1,\beta_2)\in\{\pm\}^2$, interpret $\pm$ as $\pm 1$ and let $\sign\,\vecbeta:=\beta_1\beta_2$.
Now define
\begin{subequations}
\label{e.B}	
\begin{align}
	&B_{\omega}
	=
	\prod_{i\in[m,n]_{\omega}} 
	\int_{t_{i-1}<s_i<s'_i<t_i} \d s_{i} \, \d s'_{i}
	\ \cdot \
	\frac{1}{ s_{m}} \jfn(s'_{m}-s_{m})
\\
	&\cdot
	\prod_{i\in(m,n]_{\omega}} 
	\sum_{\vecbeta_{i_-}} 
	\frac{\sign\,\vecbeta_{i_-}}{s_{i}-u_{i_-}} \jfn(s'_{i}-s_{i})
	\cdot
	\Big( \sum_{\vecbeta_{(n+1)_-}} \sign\, \vecbeta_{(n+1)_-} \Big)^{\ind_{n\in\omega}}
	\frac{1}{4\pi(s_{n+1}-u_{(n+1)_-})}\ ,
\end{align} 
\end{subequations}
with the convention that $s_{n+1}:=t_n+t_nr^2$, where the sums run over $\vecbeta_{i_-},\vecbeta_{(n+1)_-}\in\{+,-\}^2$ as said, and $u_{i}=u_{i}(\vecbeta_{[m,i]_{\omega}})$ is given by the following iteration together with $v_{i} = v_{i}(\vecbeta_{[m,i]_{\omega}})$.
For $u_i$ with $i\in[m,n)_\omega$ and $v_{i}$ with $i_-\in[m,n)_\omega$, the iteration reads
\begin{align}
	\label{e.ui}
	u_{i} =
	\begin{cases}
		s'_{i}\ , 	&	\vecbeta_i = (+,+)\ , \\
		0\ ,		&	\vecbeta_i = (-,-)\ , \\
		\frac{s'_{i} + v_{i_-}}{4}\ ,
		&	\sign\,\vecbeta_i = -1\ ,
	\end{cases}
	\qquad
	v_{i} =
	\begin{cases}
		v_{i_-}\ ,		&	\vecbeta_i = (+,+)\ , \\
		0\ ,			&	\vecbeta_i = (-,-)\ , \\
		\frac{s_{i_+}(s'_{i}+v_{i_-})}{4s_{i_+}-s'_{i}-v_{i_-}}
		&	\sign\,\vecbeta_i = -1\ ,
	\end{cases}	
\end{align}
with $v_{m_-}:=0$; when $n\notin\omega$, $u_{(n+1)_-}=u_{n}:=s'_{n}$.

\begin{proof}[Proof of Lemma~\ref{l.DZz.int}]
We begin by developing an expansion of the left hand side of \eqref{e.Dzz=sumB}.
Recall (from Section~\ref{s.prelim.tools}) that $\E Z_{s,t}=\heatsg_{s,t}$ and that $W_{s,t}:= Z_{s,t}-\heatsg_{s,t}$.
Write $\heatsg_{i,j}:=\heatsg_{t_i,t_j}$ with slight abuse of notation.
We separate the means of $D_i$ and $Z_n$ from themselves by writing
\begin{align}
	&D_i = E_i + \Ebar_i\ ,
	&&
	E_i := W_i - (W_i \Rdot 1)\, \hk_i\ ,
	&&
	\Ebar_i := \E D_i = \heatsg_{i-1,i} - \hk_i\ ,
	&&
	i\in [m,n)\ ,
\\
	&Z_n = E_n + \Ebar_n\ ,
	&&
	E_n := W_n\ ,
	&&
	\Ebar_n
	:=
	\E Z_n = \heatsg_{n-1,n}\ ,&&&
\end{align}
where the $\hk_i$s are interpreted as $\d y\, \d y' \hk(t_i,y')\in\Msp_+(\R^4)$.
Write $\hk_{n,r,y}=\hk(t_nr^2,\cdott-y)$ to simplify notation, note that $\hk_{m-1}\Ldot \Ebar_{m}=0$, and write
\begin{align}
	\label{e.DZz.expansion.}
	\DZz_{I,r}(y)
	=
	\hk_{m-1} \Ldot E_{m} \Cdot (E_{m+1}+\Ebar_{m+1}) \Cdot \cdots \Cdot (E_{n}+\Ebar_{n}) \Rdot \hk_{n,r,y}\ .
\end{align}
Expanding the last expression gives
\begin{align}
	\label{e.DZz.expansion..}
	\DZz_{I,r}(y)	
	=
	\sum_{\omega\subset(m,n]} A_{\omega}(y)\ ,
	\qquad
	A_{\omega}(y)
	:=
	\hk_{m-1} \Ldot E_{m} \Cdot \cdots \Rdot \hk_{n,r,y}\ ,
\end{align}
where the factors in the product are $E_{i}$ when $i\notin\omega$ and $\Ebar_{i}$ otherwise.
Observe that $\E A_{\omega}A_{\omega'}=0$ whenever $\omega\neq\omega'$.
To see why, assume without loss of generality $\omega\setminus\omega'\neq\emptyset$, take $i\in\omega\setminus\omega'$, and take the conditional expectation of $A_{\omega}A_{\omega'}$ given $\sigma\{Z_{s,t}\,|\,[s,t]\subset [0,t_{i-1}]\cup[t_{i},1]\}$ using the fact that $\E E_{i}=0$.
Taking the second moment of the sum in \eqref{e.DZz.expansion..} and integrating the result over $y$ give
\begin{align}
	\label{e.DZz.expansion}
	\E\,\DZz_{I,r}(y)^2
	=
	\sum_{\omega\subset(m,n]} \int_{\R^2} \d y\, \E A_{\omega}(y)^2\ .
\end{align}

Next, we express the second expectation in \eqref{e.DZz.expansion} as an integral.
Write $\omega=[i_1,j_1]\cup\ldots\cup [i_k,j_k]$ with $j_{h}+1<i_{h+1}$, use the readily verified identity 
$
	\E D_{i} \Cdot \cdots \Cdot \E D_{j}
	=
	\heatsg_{i-1,j} - \hk_{j}
$
for $[i,j]=[i_1,j_1],\ldots,[i_k,j_k]$.
Then, let $E_i^+ := W_i$ and $E^-_i := (W_i\Rdot 1) \hk_{i}$, write $E_i=E_i^+-E_i^-$ for each $i\in[m,n)_{\omega}$, and note that $E^-_{i-1}\Cdot (\heatsg_{i-1,j} - \hk_{j} )=0$.
Doing so gives
\begin{align}
	\label{e.A.expansion}
	A_\omega(y)
	=
	\hk_{m-1} \Ldot 
	\Big( \prod_{h=1}^k \cdots  E_{i_{h}-1}^+ \Cdot \big( \heatsg_{i_h-1,j_{h}} - \hk_{j_{h}} \big) \Big)
	\cdots
	\Rdot \hk_{n,r,y} \ ,
\end{align}
where the first $\cdots$ consists of products of $E_{i}^+-E^-_{i}$ with $i\in(j_{h-1},i_{h})$ and $j_{0}:=m-1$, the second $\cdots$ consists of $E_{i}^+-E^-_{i}$ with $i\in(j_{k},n)$ followed by $W_{n}$ when $j_k<n$, and should be omitted altogether when $j_k=n$.
Also, when $\omega=\emptyset$, the terms within the parentheses should be omitted altogether.
Square \eqref{e.A.expansion}, expand the result, take $\E$ with the aid of \eqref{e.Wmomformula} for $n=2$, \eqref{e.sgg2}, and relations like \eqref{e.dot.underE}, and simplify the result using the heat semigroup property.
Doing so gives
\begin{subequations}
\label{e.A.integral}
\begin{align}
	\E & A_{\omega}(y)^2
	=
	\prod_{i\in[m,n]_{\omega}}
	\int_{t_{i-1}<s_i<s'_i<t_{i}} \d s_{i} \, \d s'_{i} \int_{\R^4} \d y_{i}\, \d y'_{i}
	\cdot
	\hk(s_{m},y_{m})^2
\\
	\label{e.D.bd..integrals.2}
	&\cdot
	\prod_{i\in[m,n]_{\omega}}
	4\pi \jfn(s'_{i}-s_{i}) 
	\hk(\tfrac{s'_{i}-s_{i}}{2},y'_{i}-y_{i}) 
	\sum_{\vecbeta_i}\sign\,\vecbeta_i
	\prod_{j=1,2} f^{\beta_{ij}}(s'_{i},y'_{i},s_{i_+},y_{i_+})\ .
\end{align}
\end{subequations}
Here $i_+=i_+(\omega) :=\min (i,n+1]_{\omega}$, with the convention that $s_{n+1}:=t_n+t_nr^2$ and $y_{n+1}:=y$.
For $i\in[m,n)_{\omega}$, the sum runs over $\vecbeta_i\in\{+,-\}^2$. 
When $n\notin\omega$, the sum for $\vecbeta_{n}$ runs over the \emph{single} element $(+,+)$.
The function $f^{\beta}$ is defined as
\begin{align}
	f^{\beta}(s',y',s,y)
	:=
	\begin{cases}
		\hk(s-s',y-y')\ ,	&\text{when } \beta=+\ ,
		\\
		\hk(s,y)\ ,		&\text{when }\beta=-\ .
	\end{cases}
\end{align}
For $i\in[m,n)_{\omega}$, the vector $\vecbeta_{i}$ parameterizes the choice between $E^{+}_{i}$ and $-E^{-}_{i}$ in \eqref{e.A.expansion} when $i_+ = i+1$ and the choice between $E^+_{i}\Cdot\heatsg_{i,\,i_+-1}$ and $-E^+_{i}\Cdot\hk_{i_+-1}$ when $i_+ > i+1$ (which means $i=i_{h}-1$ and $i_+-1=j_h$ for some $h$).
When $n\notin\omega$, the second last factor in \eqref{e.A.expansion} is $W_n$, which unlike other terms is not a difference and results in the single value $(+,+)$ for $\vecbeta_{n}$.

To complete the proof, we integrate \eqref{e.A.integral} over $y$ and evaluate the spatial integrals in the order of $y_m,y'_m,y_{m+1},\ldots$.
First, write $\hk(s_m,y_m)^2\hk(\tfrac{s'_{m}-s_{m}}{2},y'_{m}-y_{m}) =\frac{1}{4\pi s_m} \hk(\frac{s_m}{2},y_m)\hk(\tfrac{s'_{m}-s_{m}}{2},y'_{m}-y_{m})$ and integrate over $y_m$ to get $\frac{1}{4\pi s_m}\hk(\frac{s'_m}{2},y'_{m})$.
Next, for the integral over $y'_{m}$, the relevant integrands are $\hk(\frac{s'_m}{2},y'_{m})$ and the last product in \eqref{e.D.bd..integrals.2} for $i=m$.
Perform the integration over $y'_{m}$ with the aid of the heat semigroup property and simplify the result.
Doing so gives
\begin{align}
	\int_{\R^{2}} \d y'_{m}\, \hk(\tfrac{s'_m}{2},y'_{m})
	\prod_{j=1,2} f^{\beta_{mj}}(s'_{m},y'_{m},s_{m_+},y_{m_+})
	=
	\frac{1}{4\pi(s_{m_+}-u_{m})} \hk(\tfrac{s_{m_+}-v_{m}}{2},y_{m_+}) \ .
\end{align}
Proceed inductively to integrate over $y_{m_+},y'_{m_+},\ldots$, note that $y_{n+1}:=y$, and reindex the result $(i,i_+)\mapsto (i_-,i)$.
Doing so gives the desired result.
\end{proof}

The fact that $u_{i}$ depends on $v_{i_-}$ and the latter depends on many $\vecbeta_j$s is inconvenient.
On the other hand, it is not hard to check from the iteration in \eqref{e.ui} that $v_{i} \leq s'_{i}$, and recall that the $t_i$s follow the exponential time scale \eqref{e.ti}.
Hence, as long as $\bb$ is small, the contribution of $v_{i_-}$ in \eqref{e.B} should be small.
The following lemma quantifies this fact.
Let
\begin{align}
	\label{e.wi}
	w_i(\vecbeta)
	:=
	\begin{cases}
		s'_{i}\ ,&\quad \vecbeta=(+,+)\ , \\
		0\ ,&\quad \vecbeta=(-,-)\ , \\
		s'_{i}/4\ ,&\quad \sign\,\vecbeta=-1\ . 
	\end{cases}
\end{align}
\begin{lem}\label{l.ui}
Setup as above.
There exists $c$ such that for all $\bb\leq 1/c$ and $i\in(m,n]_\omega$,
\begin{align}
	\label{e.ui.expansion}
	\frac{1}{s_{i}-u_{i_-}(\vecbeta_{[m,i_-]_{\omega}})}
	=
	\frac{1}{s_{i}-w_{i_-}(\vecbeta_{i_-})}
	\sum_{j\in[m,i_-]} u_{\lambda(i,j)}\big(\vecbeta_{\lambda(i,j)_\omega}\big)\ ,
	\qquad
	\lambda(i,j) := [j,i_-)\ ,
\end{align}
where $u_{\lambda(i,i_-)}=u_{\emptyset}:=1$, and the remaining $u$s satisfy $|u_{\lambda}|\leq (c\bb^2)^{|\lambda|}$ for $\lambda=\lambda(i,j)$.
\end{lem}
\begin{proof}
To begin, we use induction on $i\in[m,n]_\omega$ to show that
\begin{align}
	\label{e.vi.induction}
	v_i(\vecbeta_{[m,i]_{\omega}})
	=
	\sum_{j\in[m,i]_\omega} v_{ij}(\vecbeta_{[j,i]_{\omega}})\ ,
	\qquad
	|v_{ij}| \leq t_j\ .
\end{align}
To encode the iteration in \eqref{e.ui}, let $f(s',s,\vecbeta,v):= v\ind_{\vecbeta=(+,+)} + s(s'+v)\ind_{\sign\,\vecbeta=-1}/(4s-s'-v)$.
For $i=m$, simply take $v_{mm}(\vecbeta_m)=f(s'_m,s_{m_+},\vecbeta_m,0)$.
For $i>m$, use the iteration in \eqref{e.ui} and induction hypothesis to write
\begin{align}
	v_i(\vecbeta_{[m,i]_{\omega}})
	=
	f\Big( s'_{i},s_{i_+},\vecbeta_{i}, \sum_{j\in[m,i_-]_\omega} v_{i_-\, j}(\vecbeta_{[j,i_-]_{\omega}}) \Big)\ .
\end{align}
When $\vecbeta_i=(+,+)$ or $(-,-)$, the statement \eqref{e.vi.induction} indeed follows.
When $\vecbeta_i=(+,-)$ or $(-,+)$, Taylor expand the right hand side to get
\begin{align}
	\label{e.vi.taylor}
	\frac{s_{i_+}s'_{i}}{ 4s_{i_+}-s'_{i} }
	\Big( 1 + \frac{\sum_{j\in[m,i_-]_\omega} v_{i_-\, j}(\vecbeta_{[j,i_-]_{\omega}}) }{s'_i} \Big)
	\sum_{k\geq 0} \Big( \frac{\sum_{j\in[m,i_-]_\omega} v_{i_-\, j}(\vecbeta_{[j,i_-]_{\omega}})}{4s_{i_+}-s'_{i}} \Big)^{k}\ .
\end{align}
We want \eqref{e.vi.taylor} to be the right hand side of \eqref{e.vi.induction} for $\vecbeta_i=(+,-)$ or $(-,+)$.
Expand \eqref{e.vi.taylor}.
In the result, let $(s_{i_+}s'_{i})/(4s_{i_+}-s'_{i})$ be $v_{ii}(+,-)=v_{ii}(-,+)$.
Then, collect all the remaining terms that depend only on $\vecbeta_{[i_-,i_-]}$ and let $v_{ii_-}(\vecbeta_{i_-},+,-)=v_{ii_-}(\vecbeta_{i_-},-,+)$ be the sum of those terms.
Progress similarly to define $v_{ij}$ in decreasing order of $j$.
Given that $s'_{i} < t_{i} < s_{i_+}$ and given the exponential scale of the $t_j$s in \eqref{e.ti}, by taking $\bb$ small enough (independent of $i$) and using the induction hypothesis $|v_{i_-\,j}| \leq t_j$, it is straightforward (though tedious) to check that $|v_{ij}| \leq t_j$.

Having obtained \eqref{e.vi.induction}, let us use it to prove the desired statement.
In the definition of $u_i$ in \eqref{e.ui}, Taylor expand in $v_{i_-}$, insert \eqref{e.vi.induction} into the result, and expand accordingly.
From the result, construct $u_{[j,i_{-})_{\omega}}$ in decreasing order of $j$ similarly to how we constructed $v_{ij}$, starting with $u_{[i_-,i_-)}:=1$.
Given that $s'_{i_-} < t_{i_-} < s_{i}$ and by taking $\bb$ small enough (independent of $i$), it is straightforward (though tedious) to check that the remaining $u$s satisfy the claimed bound.
\end{proof}

Equipped with Lemma~\ref{l.ui}, we next bound the sums over the $\vecbeta$s in \eqref{e.ui}, thereby giving a bound on $|B_{\omega}|$ that no longer has the sums over the $\vecbeta$s.
Put $\inull:=(n+1)_-:=\max[m,n+1)\setminus\omega$ to simplify notation.
For $\lambda\subset [m,\inull_-)_{\omega}$, let 
\begin{align}
	\label{e.g.omega.lambda}
	&\tilg_{\omega,\lambda, i_-}(s',s)
	:=
	\begin{cases}
		\frac{1}{s-s'} - \frac{1}{s}\ ,
	\\
		\frac{1}{s-s'}\ ,
	\\
		\frac{1}{s-s'}\ ,		
	\end{cases}
	\
	g_{\omega,\lambda, i_-}
	:=
	\tilg_{\omega,\lambda, i_-}
	\begin{cases}
		1\ ,				& i_- \in [m,n)_{\omega}\setminus\lambda\ ,
	\\
		\bb^{2(i-i_-)}\ ,	& i_- \in \lambda\ ,
	\\
		1\ ,	& i_-=n \text{ or } i =m\ ,
	\end{cases}	
\\
	\label{e.G.lambda}
	&G_{\omega,\lambda}
	:=
	\prod_{i\in[m,n]_{\omega}} 
	\int_{t_{i-1}<s_i<s'_i<t_i} \d s_{i} \, \d s'_{i} \
	g_{\omega,\lambda,i_-}(s'_{i_-},s_{i})\,\jfn(s'_{i}-s_{i})
	\cdot
	g_{\omega,\lambda,\inull}(s'_{\inull},s_{n+1}) \ ,
\end{align} 
with the convention that $s'_{m_-}:=0$ and $s_{n+1}:=t_n+t_nr^2$.

\begin{lem}\label{l.B.bd.}
There exists $c$ such that, for $\omega\subset [m,n]\subset[1,N-1]$ with $m<n$ and $\bb\leq 1/c$,
\begin{align}
\label{e.B.bd.}
	|B_{\omega}|
	\leq
	c^{n-m+1} \sum_{\lambda\subset[m,\inull_-)_{\omega}} G_{\omega,\lambda}\ .
\end{align}
\end{lem}
\begin{proof}
We begin with some expansions.
Take product of the sum in \eqref{e.ui.expansion} over $i\in[m,n]_{\omega}=[m,\inull]_{\omega}$ and expand the result to get
\begin{align}
	\label{e.distributing}
	\prod_{i\in[m,n]_{\omega}} \sum_{j\in[m,i_-]_{\omega}} u_{\lambda(i,j)}
	=
	\sum_{\delta\in\Delta} U_{\delta}\ ,
	\qquad
	U_{\delta} := \prod_{(i,j)\in\delta} u_{\lambda(i,j)}\ ,
\end{align}
where $\Delta$ consists of sets of the form $\delta=\{ (i,j_i) \, |\, i\in[m,n]_{\omega}\ , j_i\in[m,i_-]_{\omega} \}$.
Let $\Lambda(\delta):=\bigcup_{(i,j)\in\delta}\lambda(i,j)$ so that $U_{\delta} = U_{\delta}(\vecbeta_{\Lambda(\delta)_{\omega}})$.
Since $u_{\lambda(i,j)}=u_{\lambda(i,j)}(\vecbeta_{\lambda(i,j)_{\omega}})$, we have $U_{\delta} = U_{\delta}(\vecbeta_{\Lambda(\delta)_{\omega}})$.
Now, insert \eqref{e.ui.expansion} into \eqref{e.B} and expand the result according to \eqref{e.distributing}.
Recalling that $\inull:=(n+1)_-$ and letting
\begin{subequations}
\label{e.B'.expansion}
\begin{align}
	B_{\omega,\delta}
	:=&
	\prod_{i\in[m,n]_{\omega}} 
	\int_{t_{i-1}<s_i<s'_i<t_i} \d s_{i} \, \d s'_{i}
	\cdot
	\frac{1}{4\pi s_{m}} \jfn(s'_{m}-s_{m})
\\
	&\cdot
	\prod_{i\in(m,n]_{\omega}} 
	\sum_{\vecbeta_{i_-}} 
	\frac{\sign\,\vecbeta_{i_-}}{s_{i}-w_{i_-}(\vecbeta_{i_-})} \jfn(s'_{i}-s_{i})	
\\
	&\cdot
	\Big(\sum_{\vecbeta_{\inull}}\sign(\vecbeta_{\inull})\Big)^{\ind_{n\in\omega}}\frac{1}{{t_n+t_nr^2-w_{\inull}(\vecbeta_{\inull})}} \,
	U_{\delta}(\vecbeta_{\Lambda(\delta)_{\omega}})\ ,
\end{align} 
\end{subequations}
we have $B_{\omega}= \sum_{\delta\in\Delta} B_{\omega,\delta}$ and in particular $|B_{\omega}|\leq \sum_{\delta\in\Delta} |B_{\omega,\delta}|$.
Let us decompose the last sum according to $\Lambda(\delta)_{\omega}$.
Note that $[m,n]_{\omega}=[m,\inull]_{\omega}$ because $\inull=(n+1)_-$\ .
Since $\Lambda(\delta)$ is a union of sets of the form $\lambda(i,j):=[j_i,i_-)$ with $i\in [m,n]_{\omega}=[m,\inull]_{\omega}$, the set $\Lambda(\delta)_{\omega}$ is contained in $[m,\inull_-)_{\omega}$.
Accordingly, we write
\begin{align}
	\label{e.B.B'.bd}
	|B_{\omega}|\leq \sum_{\delta\in\Delta} |B_{\omega,\delta}|
	=
	\sum_{\lambda\subset[m,\inull_-)_{\omega}}\sum_{\delta:\Lambda(\delta)_{\omega}=\lambda} |B_{\omega,\delta}|\ .
\end{align}

Let us bound the inner sum in \eqref{e.B.B'.bd}.
Take $\lambda$ and $\delta$ with $\Lambda(\delta)_{\omega}=\lambda$.
In \eqref{e.B'.expansion} and for $i_-\notin\lambda$, the only term that depends on $\vecbeta_{i_-}$ is $w_{i_-}$, so the sum over $\vecbeta_{i_-}$ can be evaluated independently of other factors in \eqref{e.B'.expansion}. 
Evaluate the sum and bound the absolute value of the result by the expression $ 1/s_i+(1-2)/(s_i-s'_{i_-})=\tilg_{\omega,\lambda,i_-}(s'_{i_-},s_i)$ (see \eqref{e.wi} and \eqref{e.g.omega.lambda}).
Next, for every $i_-\in\lambda\cup\{n\}$, bound $|\sign\,\vecbeta_{i_-}|\leq 1$ and $|1/(s_i-u_{i_-})| \leq 1/(s_i-s'_{i_-})=\tilg_{\omega,\lambda,i_-}(s'_{i_-},s_i)$.
Finally, note that $1/s_m=\tilg_{\omega,\lambda,m_-}(s'_{m_-},s_{m})$ by definition.
Let $\tilG_{\omega,\lambda}$ be obtained by replacing every $g$ in \eqref{e.G.lambda} with $\tilg$.
We have $|B_{\omega,\delta}| \leq 4^{n-m+1}\,\tilG_{\omega,\lambda}\,|U_{\delta}(\vecbeta_{\lambda})|$, and therefore
\begin{align}
	\label{e.B.B'.bd.}
	\sum_{\delta:\Lambda(\delta)_{\omega}=\lambda} |B_{\omega,\delta}| 
	\leq 
	4^{n-m+1}\,\tilG_{\omega,\lambda} \sum_{\delta:\Lambda(\delta)_{\omega}=\lambda} |U_{\delta}(\vecbeta_{\lambda})|\ .
\end{align}

To complete the proof, let us bound the last sum in \eqref{e.B.B'.bd.}.
First, note that since $g$ and $\tilg$ differs only by the multiplicative factors in \eqref{e.g.omega.lambda}, letting $k_{\lambda}:=\sum_{i_-\in\lambda}(i-i_-)$, we have $\tilG_{\omega,\lambda}=\bb^{-2k_{\lambda}}G_{\omega,\lambda}$.
Next, write $\Lambda(\delta)_{\omega}=\lambda$ more explicitly as
$
	\bigcup_{(i,j)\in\delta} [j,i_-)_{\omega} = \lambda
$
and observe that dropping the $\omega$ subscript (whereby increasing the set $[j,i_-)_{\omega}$ to $[j,i_-)$) leads to
\begin{align}
	\label{e.union.relation}
	\bigcup_{(i,j)\in\delta} [j,i_-) \supset \bigcup\big\{ [h_-,h) \, \big| \, h_-\in\lambda \big\}\ ,
\end{align}
because every $h_-\in\lambda=\Lambda(\delta)_{\omega}$ belongs to $[j,i_-)$ for some $(i,j)\in\delta$ and $h_-<i_-$ implies $h\leq i_-$.
The set on the right hand side of \eqref{e.union.relation} has $k_{\lambda}$ elements.
Given \eqref{e.union.relation} and the bound $|u_{\lambda}|\leq (c\bb^2)^{|\lambda|}$ from Lemma~\ref{l.ui}, referring back to the definition of $U_{\delta}$ in \eqref{e.distributing}, we have
\begin{align}
	\sum_{\delta:\Lambda(\delta)_\omega=\lambda} |U_{\delta}|
	=
	\sum_{\delta:\Lambda(\delta)_\omega=\lambda} \prod_{(i,j)\in\delta}|u_{\lambda(i,j)}|
	\leq
	\prod_{i\in[m,n]_{\omega}}\sum_{k_i\geq 0} (c\bb^2)^{k_i} \cdot \ind_{\sum k_i \geq k_{\lambda}}\ ,
\end{align}
where we use $k_i:=i_--j_i$ to parameterize the cardinality of $\lambda(i,j)$.
The right hand side evaluates to $\sum_{k\geq k_{\lambda}} \binom{|[m,n]_{\omega}|+k-1}{k} a^{k}$ for $a=c\bb^2$ and is bounded by $2^{n-m+1} (2a)^{k_\lambda}$ for $|a| \leq 1/2$.
Insert this bound into \eqref{e.B.B'.bd.}, substitute in $\tilG_{\omega,\lambda}=\bb^{-2k_{\lambda}}G_{\omega,\lambda}$ and insert the result into \eqref{e.B.B'.bd}.
Doing so gives the desired result~\eqref{e.B.bd.}.
\end{proof}

Next, we proceed to bound the time integrals in $G_{\omega,\lambda}$.
The bound in Lemma~\ref{l.B.bd.} took care of the sums over the $\vecbeta$s, but the time integrals remained in the bound, as seen in the definition \eqref{e.G.lambda} of $G_{\omega,\lambda}$.
Our next step is to bound the time integrals.
To reiterate, we have assumed $|\theta|\leq c_0$ and will mostly omit showing the dependence on $\theta$ or $c_0$.
In particular, $G_{\omega,\lambda}=G_{\omega,\lambda}^{\theta}$ and the $c$ in \eqref{e.B.bd} depends on $c_0$.

\begin{lem}\label{l.B.bd}
There exists $c$ such that, for $\omega\subset [m,n]\subset[1,N-1]$ with $m<n$, $\lambda\subset [m,\inull_-)_{\omega}$, $\bb\leq 1/c$, and $r\in(0,\infty)$,
\begin{align}
\label{e.B.bd}
	G_{\omega,\lambda}
	\leq
	c^{n-m+1}\,
	\prod_{i\in [m,n]_{\omega}} \frac{1}{\log t_i^{-1}}
	\cdot
	\frac{ \bb^{2|\omega|}\, \log \bb^{-1} }{t_n(1+r^2)^{2\ind_{n\in\omega}}}
	\Big(\log\frac{1+r^2}{r^2}\Big)^{\ind_{n\notin\omega}}
	\ .	
\end{align}
\end{lem}
\begin{proof}
The proof amounts to bounding the integrals in \eqref{e.G.lambda}.
We consider $m+1<n$ only; the case $m+1=n$ can be handled the same way. 
Recall that $\inull:=(n+1)_-:=\max[m,n]_{\omega}$.

We start by bounding the last integral in \eqref{e.G.lambda}, namely the $i=\inull$ integral
\begin{align}
	\label{e.G.lambda.inull}
	\int_{t_{\inull-1}<s<s'<t_{\inull}} \d s \d s' \
	g_{\omega,\lambda,\inull_-}(s'_{\inull_-},s)\,\jfn(s'-s)\,g_{\omega,\lambda,\inull}(s',t_n+t_nr^2) \ ,
\end{align}
where we renamed the dummy variables $s_{\inull},s'_{\inull}$ to $s,s'$ and substituted in $s_{n+1}:=t_n+t_nr^2$.
Let $u:=s-t_{\inull-1}$, $u':=s'-s$, and $u'':= t_{\inull}-s'$, take any $v\in[1/2,1]$, and perform the change of variables $u,u',u''\mapsto \frac{1}{v}u, \frac{1}{v}u',\frac{1}{v}u''$ to get
\begin{align}
	\label{e.G.lambda.inull.}
	\eqref{e.G.lambda.inull}
	=
	\int \frac{\d u \d u'}{v^2} \
	g_{\omega,\lambda,\inull_-}\big(s'_{\inull_-},\tfrac{1}{v}u+t_{\inull-1}\big) \,
	\jfn\big(\tfrac{1}{v}u'\big) \,
	g_{\omega,\lambda,\inull}\big(t_{\inull}-\tfrac{1}{v}u'',t_n+t_nr^2\big) \ ,
\end{align}
where the integral runs over $u+u'+u''=v(t_{\inull}-t_{\inull-1})$.
Let us bound the three factors in the integrand (after $1/v^2$).
Given that $v \geq 1/2$, using \eqref{e.g.omega.lambda} for $i_-=\inull_-$ and combining the two fractions there when $\inull_-\notin\lambda\cup\{n\}$, we see that the first factor is bounded by $4 g_{\omega,\lambda,\inull_-}(\frac{1}{2}s'_{\inull_-},u+\frac{1}{2}t_{\inull-1})$.
As for the second factor, referring back to \eqref{e.jfn}, we see that $\jfn^\theta(\frac{1}{v}u')=v\,\jfn^{\theta-\log v}(u') \leq \jfn^{\theta+\log 2}(u')$.
To bound the last factor, consider the cases $n\in\omega$ and $n\notin\omega$ separately.
Given that $\inull\notin\lambda$ and that $\inull<n$ in the former case, we use the first expression in \eqref{e.g.omega.lambda} for $i_-=\inull$, combine the two fractions there, and use $s'_{\inull} \leq t_{\inull} = t_{n}\bb^{2n-2\inull}$ and $\inull<n$.
Doing so shows that the last factor in \eqref{e.G.lambda.inull} is bounded by $c\,\bb^{2n-2\inull}/t_n(1+r^2)^{2}$.
In the latter case $n\notin\omega$, we have $\inull=n$.
Using the last expression in \eqref{e.g.omega.lambda} for $i_-=\inull=n$ and using $v\leq 1$ show that the last factor in \eqref{e.G.lambda.inull} is bounded by $1/(t_nr^2+u'')$.
Combining these bounds and bounding the $1/v^2$ in \eqref{e.G.lambda.inull.} by $4$, we arrive at
\begin{align}
	\label{e.G.lambda.inull..}
	\eqref{e.G.lambda.inull}
	\leq
	c\, \bb^{2n-2\inull}
	\int \d u \d u'\
	g_{\omega,\lambda,\inull_-}\big(\tfrac{s'_{\inull_-}}{2},u+\tfrac{t_{\inull-1}}{2}\big) \,
	\jfn^{\theta+\log 2}(u') \,
	\begin{cases}
		\frac{1}{t_n(1+r^2)^{2}} & n\in\omega\ ,
	\\
		\frac{1}{t_nr^2+u''} & n\notin\omega\ ,
	\end{cases}
\end{align}
where the integral runs over $u+u'+u''=v(t_{\inull}-t_{\inull-1})$.
This bound holds for all $v\in[1/2,1]$.
Let us average it over $v\in[1/2,1]$.
In the averaged bound, perform the change of variables $(u,u',v)\mapsto (u,u',u''')$ where $u''':=v(t_{\inull}-t_{\inull-1})-u-u'$, bound the resulting Jacobian as $(t_{\inull}-t_{\inull-1})^{-1}\leq c t_{\inull}^{-1}$, release the range of integration to $u,u',u'''\in[0,t_{\inull}]$ so that the resulting integral \emph{factorizes}.
Evaluating and bounding the factorized integral with the aid of \eqref{e.jfn.intbd} and noting that $\inull_-\notin\lambda$ show that 
\begin{align}
	\label{e.B.bd.last}
	\eqref{e.G.lambda.inull}
	\leq
	&c\,
	\log \frac{1-s'_{\inull_-}/(2t_{\inull}+t_{\inull-1})}{1-s'_{\inull_-}/t_{\inull-1}} 
	\frac{\bb^{2(n-\inull)}}{t_n(1+r^2)^{2\ind_{n\in\omega}}} \Big( \log \frac{r^2+1}{r^2} \Big)^{\ind_{n\notin\omega}}\ .
\end{align} 

We now complete the proof by evaluating and bounding the remaining integrals in \eqref{e.G.lambda}.
The following bounds hold.
\begin{align}
	\label{e.B.bd.iter0}
	&\log \frac{1-s'_{i_-}/(at_{i}+t)}{1-s'_{i_-}/t_{i-1}} + \frac{s'_{i_-}}{t_{i}} 
	\leq
	c\, \bb^{2(i-1-i_-)}\log \frac{2}{1-s'_{i_-}/t_{i-1}}\ ,
	\quad
	a\geq 1\ ,\ t\geq 0\ ,
\\
	\label{e.B.bd.iter1}
	\begin{split}
	&\int_{t_{i-1}<s<s'<t_{i}} \d s \d s' \
	g_{\omega,\lambda,i_-}(s'_{i_-},s) \jfn(s'-s) \log\frac{2}{1-s'/t_{i}}
\\
	&\qquad\leq\frac{c}{\log t_{i}^{-1}}
	\begin{cases}
		\log\frac{1-s'_{i_-}/t_{i}}{1-s'_{i_-}/t_{i-1}} + \frac{s'_{i_-}}{t_{i}}\ , & i_- \in [m,n)_{\omega}\setminus\lambda\ ,
	\\		
		\bb^{2(i-i_-)}\log\frac{2}{\bb^2-s'_{i_-}/t_{i}}
		\ , & i_-\in\lambda\ ,
	\\		
		\log\bb^{-1}\ , & i=m\ .
	\end{cases}
	\end{split}
\end{align}
The first bound follows by considering $i_-<i-1$ and $i_-=i-1$ separately, bounding the left hand side by $c\,\bb^{2(i-1-i_-)}$ in the first case, and bounding $1-s'_{i_-}/(at_{i}+t)\leq 1$ in the second case.
The second bound follows by bounding the integral over $s'$ first, dividing the domain into $s<s'<(t_i+s)/2$ and $(t_i+s)/2<s<t_i$ and using the upper bound in \eqref{e.jfn.estimate} with $\epsilon_+(c_0,t)\leq c(c_0)$, and bounding the integral over $s'$ later, again dividing the domain into $t_{i-1}<s<(t_{i-1}+t_{i})/2$ and $(t_{i-1}+t_{i})/2<s<t_i$.
Next, for the case $i_-\in\lambda$ in \eqref{e.B.bd.iter1}, write $\bb^2-s'_{i_-}/t_{i}=\bb^{2}(1-s'_{i_-}/t_{i-1})$ and use $\bb^{2}\log\bb^{-2}\leq c$ to bound the expression in the brace together by $c\,\bb^{2(i-1-i_-)}\log 2/(1-s'_{i_-}/t_{i-1})$.
Combining the resulting bounds with \eqref{e.B.bd.iter0} and letting $\phi_i(s'):=\log 2/(1-s'/t_{i})$ to simplify notation give
\begin{align}
	\label{e.B.bd.iter1.}
	&\int_{t_{i-1}<s<s'<t_{i}} \d s \d s' \
	g_{\omega,\lambda,i_-}(s'_{i_-},s) \jfn(s'-s) \phi_{i}(s')
	\leq
	\frac{c\,\bb^{2(i-1-i_-)}}{\log t_{i}^{-1}}\phi_{i-1}(s'_{i_-})\ ,
	\ i\in(m,n)_{\omega}\ .
\end{align}
Use \eqref{e.B.bd.iter0} in \eqref{e.B.bd.last}, then use \eqref{e.B.bd.iter1.} in decreasing order of $i$, and use \eqref{e.B.bd.iter1} for $i=m$.
In the result, recognize $n-\inull+\sum_{i\in(m,\inull]_{\omega}}(i-1-i_-)$ as $|\omega|$.
Doing so gives the desired bound \eqref{e.B.bd}.
\end{proof}

We now derive the needed bounds based on previous results.
Let 
\begin{align}
	\label{e.DZz'}
	\DZz'_{I}:=\hk_{m-1}\Ldot D_{m}\Cdot\cdots\Cdot D_{n-1}\Cdot Z_{n}\ ,\qquad I=[m,n]\ .
\end{align}
This is a random signed measure whose total variation has finite positive moments, namely $\E(|\DZz'_{I}|1)^k<\infty$, as can be seen by expressing the $D$ in \eqref{e.DZz'} in terms of $Z$ and the heat kernel.
Let $\hkk(t,x_1,x_2):=\exp(-|x_1-x_2|^2/4t)$, where $x_i\in\R^2$, and note that by the heat semigroup property,
\begin{align}
	\label{e.hkk.id}
	4\pi t \int_{\R^2} \d y \prod_{j=1,2}\hk(t,x_j-y)
	=
	\hkk(t,x_1,x_2)\ .
\end{align}

\begin{cor}\label{c.DZ.bd}
There exists $c$ such that, for $I=[i_I,j_I]\subset [1,N-\ell]$ with $|I|>1$, for $\bb\leq 1/c$, and for $\ell\in[1,N]$, 
\begin{align}
\label{e.DZz.bd}
	4\pi t_{j_{I}} r^2 \int_{\R^2} \d y\, \E\,\DZz_{I,r}(y)^2
	&\leq
	\frac{ 
		c^{|I|}
	}{
	 (\log \bb^{-1})^{|I|-1}
	}
	r^2\log \frac{r^2+1}{r^2}\ ,
\\
\label{e.DZ.bd}
	\E\,\DZ_{I}{}^2
	&\leq
	\frac{ 
		c^{|I|}
	}{
	 (\log \bb^{-1})^{|I|-1} (N-j_{I})^{2}
	}\ .
\end{align}
\end{cor}

\begin{proof}
In this proof, we write $m=i_I$ and $n=j_I$ to simplify notation.

To obtain \eqref{e.DZz.bd}, multiply both sides of \eqref{e.B.bd} by $4\pi t_nr^2$, use $\frac{1}{1+r^2}\leq \log \frac{1+r^2}{r^2}$ (because $\log(1-x)\leq -x$ for $x=\frac{1}{1+r^2}$), sum both sides over $\omega\subset(m,n]$ and $\lambda\subset[m,\inull_-)_{\omega}$, insert the result into \eqref{e.Dzz=sumB}.
Doing so gives
\begin{align}
	\label{e.DZz.bd.}
	4\pi t_n r^2 \int_{\R^2} \d y\, \E\,\DZz_{I,r}(y)^2
	\leq
	c^{|I|}\log\bb^{-1} 
	\prod_{i\in I}\Big( \frac{1}{\log t_{i}^{-1}} + \frac{\ind_{i>m}\,\bb^{2}}{(1+r^2)^{\ind_{i=n}}} \Big)
	\cdot
	r^2\log \frac{r^2+1}{r^2}\ .
\end{align}
The desired bound \eqref{e.DZz.bd} follows by using 
\begin{align}
	\label{e.logti}
	\log t_i^{-1} = 2\log\bb^{-1}\cdot\Big(\frac{\log\e^{-1}}{\log\bb^{-1}}-i\Big) \geq 2\log\bb^{-1}\cdot(N-i)\ ,
	&&
	\bb^{2} \leq \frac{c}{\log\bb^{-1}}\ ,
\end{align}
and $N-i \geq 1$ for $i\in I$.

The second bound \eqref{e.DZ.bd} follows by sending $r\to\infty$ in \eqref{e.DZz.bd.}.
Send $r\to\infty$ on the right hand side of \eqref{e.DZz.bd.}, use \eqref{e.logti}, use $N-i \geq N-n=N-j_I$ for $i=m,n$, and use $N-i \geq 1$ for $i\in(m,n)$.
Doing so shows that the limit of the right hand side of \eqref{e.DZz.bd.} is bounded by that of \eqref{e.DZ.bd}.
As for the left hand side, let $f_r(y,x_1,x_2) := 4 \pi t_nr^2 \hk(t_nr^2,x_1-y)\hk(t_nr^2,x_2-y)$ and write the left hand side of \eqref{e.DZz.bd.} as $\int \d y\, \E\,\DZz'_{I}{}^{\otimes 2} \, f_r(y,\cdott,\cdott)$.
Given the property mentioned after \eqref{e.DZz'}, we swap the integral with $\E\,\DZz'_{I}{}^{\otimes 2}$, use \eqref{e.hkk.id} for $t=t_nr^2$, and send $r\to\infty$ with the aid of the dominated convergence theorem.
Doing so shows that the left hand side of \eqref{e.DZz.bd.} converges to that of \eqref{e.DZ.bd}.
\end{proof}

\section{Decoupling}
\label{s.decoupling}
The main goal in this section is to prove Proposition~\ref{p.decoupling.bd}, which asserts that the sum in \eqref{e.decoupling} is smaller than $1/2$ on the event $\Wevent$ (defined in \eqref{e.Wevent}) when $\ell \log\bb^{-1}$ is large enough.
Recall $\scrI=\scrI(N,\ell)=\scrI(N_{\e,\bb},\ell)$ from before \eqref{e.scrI}.
Let $\filt_{i}$ be the sigma algebra generated by $Z_{s,t}$ for $s\leq t\in[t_{i-1},t_{i}]$ and note that by Axiom~\ref{d.shf.inde}, $\filt_1,\ldots,\filt_{N-\ell}$ are independent.
To reiterate, we have assumed $|\theta|\leq c_0$ and will mostly omit showing the dependence on $\theta$.
\begin{prop}\label{p.decoupling.bd}
There exists $c=c(c_0)$ such that for all $\bb \leq 1/c$ and $\ell\in[1,N]$\ ,
\begin{align}
	\label{e.decoupling.bd}
	\E\Big( 
		\sum_{\calI\in\scrI}
		\prod_{I\in\calI} 
		\frac{
			\ind_{\Wevent}\ \DZz_{I}
		}{
			(\hk_{i_{I}-1} \Ldot Z_{i_{I}} \Rdot 1) \cdots (\hk_{j_{I}-1} \Ldot Z_{j_{I}} \Rdot 1) 
	}\Big)^2
	\leq
	\frac{c}{(\ell\log\bb^{-1})^{1/2}}\ .
\end{align}
\end{prop}
\begin{proof}
We begin with a preliminary expansion.
Let $F_I$ denote the fraction in \eqref{e.decoupling.bd} and write the left hand side of \eqref{e.decoupling.bd} as $\sum_{\calI,\calJ\in\scrI} \E [\prod_{I\in \calI} F_I \cdot \prod_{J\in\calJ} F_J]$.
We seek to break the expectation into a product by the independence of $\filt_{1},\ldots,\filt_{N-\ell}$.
View $1,\ldots,N-\ell$ as vertices, put an edge between two vertices if they both belong to the same $I\in\calI$ or $J\in\calJ$.
This gives a graph, which we call $G_{\calI,\calJ}$, and its connected components give a partition of $[1,N-\ell]\cap\Z$ into intervals.
The singleton intervals contain no elements in $\bigcup\calI\cup\bigcup\calJ$, because the intervals in $\calI$ and $\calJ$ have at least two elements (by \eqref{e.scrI}).
We discard the singleton intervals and let 
\begin{align}
	\label{e.calK_calI,calJ}
	\calK_{\calI,\calJ} :=
	\big\{ K \, | \, K \text{ a connected component of } G_{\calI,\calJ} \text{ with } |K| \geq 2 \big\}\ .
\end{align}
By the independence of  $\filt_{1},\ldots,\filt_{N-\ell}$, $\E[ \prod_{I\in \calI} F_I \cdot \prod_{J\in\calJ} F_J]=\prod_{K\in\calK_{\calI,\calJ}} A_{K,\calI,\calJ}$, where
\begin{align}
	\label{e.decoupling.bd.A}
	A_{K,\calI,\calJ} := \E \Big[ \prod_{I\in \calI, I\subset K} F_I \cdot \prod_{J\in\calJ, J\subset K} F_J \Big]\ ,
\end{align}
and therefore
\begin{align}
	\label{e.decoupling.bd.expansion}
	\text{lhs of } \eqref{e.decoupling.bd}
	=
	\sum_{\calI,\calJ\in\scrI} \prod_{K\in\calK_{\calI,\calJ}} A_{K,\calI,\calJ}\ .
\end{align}

Writing $c=c(c_0)$ to simplify notation, we show that the summand in \eqref{e.decoupling.bd.expansion} satisfies
\begin{align}
	\label{e.decoupling.bd.Abd}
	A_{K,\calI,\calJ} 
	\leq 
	\frac{1}{(N-j_{K})^{3/2}} \frac{c^{|K|}}{ (\log\bb^{-1})^{(|K|-1)/2} }
	\ ,
	\quad
	K = [i_K,j_K]\ .
\end{align}
In \eqref{e.decoupling.bd.A}, consider the case where neither product is empty ($\prod_{\emptyset}:=1$) and the complement case.
\begin{description}[leftmargin=5pt]
\item[Case~1, neither product in \eqref{e.decoupling.bd.A} is empty] 
Apply the Cauchy--Schwarz inequality in \eqref{e.decoupling.bd.A} and bound $(\hk_{i-1}\Ldot Z_{i} \Rdot 1 )\ind_{i}\geq 1/2$ using the indicator $\ind_{\Wevent}$ to get
$
	A_{K,\calI,\calJ}
	\leq
	c^{|K|}
	(
		\prod_{I\in\calI,I\subset K} \E\,{\DZ_I}^2 \cdot \prod_{J\in\calJ,J\subset K} \E\,{\DZ_J}^2
	)^{1/2}
$.
By the definition of $\calK_{\calI,\calJ}$, every element of $K$ is contained in at least one of the $I$s or $J$s.
By the definition of $\scrI$, every $I\in\calI$ or $J\in\calJ$ contains at least two elements.
Given these properties, applying \eqref{e.DZ.bd} leads to \eqref{e.decoupling.bd.Abd}, with a better power of $1/(N-j_{K})^{2}$ in fact.
\item[Case~2, one product in \eqref{e.decoupling.bd.A} is empty]
In this case, since the intervals in $\calI$ are disjoint and the same holds for $\calJ$, $A_{K,\calI,\calJ}=\E F_{K}$.
Recall from \eqref{e.Wevent} that $\ind_{\Wevent}=\prod_{i}\ind_{\Wevent_i}$.
Use this and $\hk_{i-1}\Ldot Z_{i} \Rdot 1=1+\hk_{i-1}\Ldot W_{i} \Rdot 1$ for $i=i_K$ to write
\begin{align}
	\label{e.dcoupling.bd.F}
	F_K
	=
	\DZ_{K} \Big( 
		1 - \Big(
			\ind_{\Wevent^\compl_{i_K}}+\frac{\hk_{i_K-1}\Ldot W_{i_K} \Rdot 1}{\hk_{i_K-1}\Ldot Z_{i_K} \Rdot 1}\ind_{\Wevent_{i_K}}
			\Big)\Big)
	\prod_{i\in(i_K,j_K]} \frac{\ind_{\Wevent_i}}{\hk_{i-1}\Ldot Z_{i} \Rdot 1}
	\ .
\end{align}
Let $\E_{i}=\E[\cdott\,|\,\ldots,\filt_{i-1},\filt_{i+1},\ldots]$ and take $\E_{i_K}$ on both sides of \eqref{e.dcoupling.bd.F}.
Using $\hk_{i-1}\Ldot \E_{i} D_{i} = \hk_{i-1}\Ldot (\heatsg_{t_{i-1},t_{i}} - \hk_{i-1}) = 0$ for $i=i_K$ shows that the first $1$ in \eqref{e.dcoupling.bd.F} does not contribute to the result.
Further taking $\E$ gives
\begin{align}
	\label{e.dcoupling.bd.EF}
	\E F_K
	=
	\E\ 
	\DZ_{K} \Big( 
		-\ind_{\Wevent^\compl_{i_K}}-\frac{\hk_{i_K-1}\Ldot W_{i_K} \Rdot 1}{\hk_{i_K-1}\Ldot Z_{i_K} \Rdot 1}\ind_{\Wevent_{i_K}}
			\Big)
	\prod_{i\in(i_K,j_K]} \frac{\ind_{\Wevent_i}}{\hk_{i-1}\Ldot Z_{i} \Rdot 1}
	\ .
\end{align}
Bound the terms within the parenthesis in absolute value by $4|\hk_{i_K-1}\Ldot W_{i_K} \Rdot 1|$, bound the last products by $2^{|K|-1}$, apply the Cauchy--Schwarz inequality with the aid of \eqref{e.DZ} (note that $|K|\geq 2$) and $\E(\hk_{i_K-1}\Ldot W_{i_K} \Rdot 1)^2 \leq c /(N-i_K)\leq c/(N-{j_K})$, the latter of which follows from \eqref{e.W.2ndmom}.
Doing so gives the claim \eqref{e.decoupling.bd.Abd}. 
\end{description}

The rest of the proof consists of combinatorial bounds.
Let $\scrK:=\{\calK_{\calI,\calJ}\,|\,\calI,\calJ\in\scrI\}$, and for each $\calK\in\scrK$ let $\scrI^2_\calK:=\{(\calI,\calJ)\in\scrI^2\,|\,\calK=\calK_{\calI,\calJ}\}$.
Insert \eqref{e.decoupling.bd.Abd} into \eqref{e.decoupling.bd.expansion} to get
\begin{align}
	\label{e.decoupling.bd.expansion.}
	\text{lhs of \eqref{e.decoupling.bd}} 
	\leq
	\sum_{\calK\in\scrK}
	\sum_{(\calI,\calJ)\in\scrI^2_{\calK}}
	\prod_{K\in\calK}
	\frac{1}{(N-j_{K})^{3/2}} 
	\frac{c^{|K|}}{ (\log\bb^{-1})^{(|K|-1)/2} }\ .
\end{align}
Bound the size of the second sum by $\prod_{K\in\calK}(\sum_{i\leq |K|/2}\binom{|K|}{2i})^2 \leq \prod_{K\in\calK} c^{|K|}$, absorb the bound into the existing $c^{|K|}$, and use $\scrK\subset\scrI$ (see \eqref{e.calK_calI,calJ} and \eqref{e.scrI}) to replace the sum over $\calK\in\scrK$ by the sum over $\calK\in\scrI$.
Next, for each $\calK\in\scrI$, given that the intervals in $\calK$ are disjoint, order them according to the order on $\Z$ as $(K_1,\ldots,K_{|\calK|})$, and define $\type\calK:=(|K_1|,\ldots,|K_{|\calK|}|)\in\Z_{\geq 2}^{|\calK|}$.
This allows us to write the sum over $\calK\in\scrI$ as $\sum_{m=2}^{N-\ell}\sum_{\vecu: \sum u_i=m}\sum_{\calK: \type\calK=\vecu}$\ , and therefore
\begin{align}
	\label{e.decoupling.bd.expansion..}
	\text{lhs of \eqref{e.decoupling.bd}} 
	\leq
	\sum_{m=2}^{N-\ell}
	\frac{c^m}{(\log\bb^{-1})^{(m-1)/2}}
	\
	\sum_{\vecu: \sum u_i=m}
	\
	\sum_{\calK: \type\calK =\vecu}
	\
	\prod_{h=1}^{|\vecu|}
	\frac{1}{(N-j_{K_{h}})^{3/2}}\ .
\end{align}
Recall that we ordered the intervals in $\calK$, so $j_{K_1}<j_{K_2}<\ldots\in[1,N-\ell]$.
Hence, the innermost sum is bounded by $c^{|\vecu|}\ell^{-|\vecu|/2}/|\vecu|! \leq c^{m}\ell^{-1/2}/|\vecu|!$\ , and $|\{\vecu\,|\, u_1+\ldots+u_{|\vecu|}=m\}|/|\vecu|!$ is bounded by the number of integer partitions of $m$, which is bounded by $c^{\sqrt{m}}\leq c^{m}$; see~\cite[Equation~(5.1.2)]{andrews1998theory} for example.
Summing over $m$ gives the desired result.
\end{proof}

Combining Propositions~\ref{p.expansion} and \ref{p.decoupling.bd} immediately gives
\begin{cor}\label{c.clt}
There exists $c=c(c_0)$ such that for all $\bb\leq 1/c$ and $\ell\in[1,N]$,
\begin{align}
	\P\Big[ \ \ind_{\Wevent}\Big|\frac{ \hk(\e^{2}) \Ldot Z_{0,t_{N-\ell}} \Rdot 1}{ \prod_{i=1}^{N-\ell} \hk_{i-1} \Ldot Z_{i} \Rdot 1 } - 1 \Big| > \frac{1}{2} \ \Big] 
	\leq 
	\frac{c}{(\ell\log\bb^{-1})^{1/2}}\ .
\end{align}
\end{cor}

The rest of this section is devoted to proving Corollary~\ref{c.quenched}, which will be used in Section~\ref{s.last.piece}.
We start by employing the same argument in the proof of Proposition~\ref{p.decoupling.bd} to derive a preliminary bound.
Recall that $\hkk(t,x_1,x_2):=\exp(-|x_1-x_2|^2/4t)$, write $\hkk(t):=\hkk(t,\cdott)$.

\begin{lem}\label{l.decoupling'}
There exists $c=c(c_0)$ such that for all $\bb\leq 1/c$, $r\in(0,\infty)$, and $\ell\in[1,N]$,
\begin{align}
	\label{e.decoupling'}
	\E\,
	\frac{%
		\ind_{\Wevent} ( \hk(\e^2) \Ldot Z_{0,t_{N-\ell}} )^{\otimes 2} \Rdot \hkk(t_{N-\ell}r^2)
	}{%
		(\prod_{i=1}^{N-\ell} \hk_{i-1} \Ldot Z_{i} \Rdot 1)^2
	}
	\leq
	c\,r^2 \log\frac{r^2+1}{r^2}\ .
\end{align}
\end{lem}
\begin{proof}
We begin with a decoupling expansion.
First, write $\hk(t_{N-\ell}r^2,\cdott-y)=\hk_{N-\ell,r,y}$ to simplify notation and use \eqref{e.hkk.id} to write
\begin{align}
	\label{e.decoupling'.1}
	\text{lhs of \eqref{e.decoupling'}}
	=
	4 \pi t_{N-\ell} r^2
	\int_{\R^2} \d y \, \E\Big(
	\frac{
		\ind_{\Wevent} \, \hk(\e^2) \Ldot Z_{0,t_{N-\ell}} \Rdot \hk_{N-\ell,r,y}
	}{%
		\prod_{i=1}^{N-\ell} \hk_{i-1} \Ldot Z_{i} \Rdot 1
	}\Big)^2\ .
\end{align}
Next, consider a set $\calI=\{I_1,\ldots,I_{|\calI|}\}$ of disjoint intervals of integers in $[1,N-\ell]$.
We require that an interval in $\calI$ contains $N-\ell$ (there is at most one because the intervals are disjoint), every other interval in $\calI$ has at least two elements, and $\cup\calI$ contains at least two elements.
Letting $\scrI'$ denote the set of all such $\calI$s and
\begin{align}
	M_{I}(y) &:= 
	\begin{cases}
		\DZ_{I}\ , & \text{ when } |I| \geq 2 \text{ and } N-\ell\notin I\ ,
	\\
		\DZz_{I,r}(y)\ , & \text{ when } |I| \geq 2 \text{ and } N-\ell\in I\ ,
	\\
		\hk_{N-\ell-1} \Ldot Z_{N-\ell} \Rdot \hk_{N-\ell,r,y}\ , & \text{ when } I = \{ N-\ell\}\ ,
	\end{cases}
\end{align}
we perform the analogous expansion of \eqref{e.decoupling.} to the numerator in \eqref{e.decoupling'.1} and divide the result by the denominator in \eqref{e.decoupling'.1} to conclude that 
\begin{align}
	\label{e.decoupling'.2}
	\text{fraction in \eqref{e.decoupling'.1}}
	=
	\ind_{\Wevent}
	\frac{M_{\{N-\ell\}}(y)}{\hk_{N-\ell-1}\Ldot Z_{N-\ell} \Rdot 1}
	+
	\sum_{\calI\in\scrI'}
	\prod_{I\in\calI}
	\frac{
		\ind_{\Wevent}\, M_I(y)
	}{
		\prod_{i\in I} \hk_{i-1}\Ldot Z_{i} \Rdot 1
	}\ .
\end{align}
Note that \eqref{e.decoupling'.2} differs from \eqref{e.decoupling} because the last test function in the numerator in \eqref{e.decoupling'.1} is $\hk_{N-\ell,r,y}$ and not $1$.
This fact is also reflected in how $\scrI'$ differs from $\scrI$ (defined in \eqref{e.scrI}).
 
Next, we bound \eqref{e.decoupling'.1} using the above decoupling expansion.
In \eqref{e.decoupling'.2}, bound the first denominator from below by $1/2$ using of $\ind_{\Wevent}$, square the result using $(x+y)^2\leq 2x^2+2y^2$, apply $4\pi t_{N-\ell} r^2\int_{\R^2} \d y \, \E$ to the result, use \eqref{e.Z.hkk} for $i=N-\ell$ to bound the first term on the right hand side of the result, and write $c=c(c_0)$ to simplify notation.
Doing so shows that 
\begin{align}
	\label{e.decoupling'.3}
	\text{lhs of \eqref{e.decoupling'}}
	\leq
	c\,r^2 \log\frac{r^2+1}{r^2}
	+
	c t_{N-\ell} r^2 \int_{\R^2} \d y \, 
	\E\Big(
		\sum_{\calI\in\scrI'}
		\prod_{I\in\calI}
		\frac{
			\ind_{\Wevent}\, M_I(y)
		}{
			\prod_{i\in I} \hk_{i-1}\Ldot Z_{i} \Rdot 1
		}
	\Big)^2\ .
\end{align}

It now suffices to bound the last term in \eqref{e.decoupling'.3}, and we do so by employing the argument in the proof of Proposition~\ref{p.decoupling.bd}.
Let $F'_{I}(y)$ denote the last fraction in \eqref{e.decoupling'.3}, for $\calI,\calJ\in\scrI'$ let $G_{\calI,\calJ}$ be the graph defined before \eqref{e.calK_calI,calJ}, and let $\calK'_{\calI,\calJ}$ be the set of connected components $K$ of $G_{\calI,\calJ}$ such that $|K|\geq 2$ or $K\ni N-\ell$.
Similar to \eqref{e.decoupling.bd.expansion}, here we have
\begin{align}
	\label{e.decoupling'.expansion}
	\text{last term in \eqref{e.decoupling'.3}}
	&=
	\sum_{\calI,\calJ\in\scrI'} \prod_{K\in\calK'_{\calI,\calJ}} A'_{K,\calI,\calJ}\ ,
\\
	\label{e.pdecoupling'.A}
	A'_{K,\calI,\calJ} &:= c t_{N-\ell} r^2 \int_{\R^2} \d y \, \E \Big[ 
		\prod_{I\in \calI, I\subset K} F'_I(y) \cdot \prod_{J\in\calJ, J\subset K} F'_J(y) 
	\Big]\ .
\end{align}
By definition, there is exactly one interval in $\calK'_{\calI,\calJ}$ that contains $N-\ell$, and the same holds for sets in $\scrI'$.
We write that interval as $K_*$.
Note that there is exactly one $I$ and one $J$ in the products in \eqref{e.pdecoupling'.A} that contain $K_*$.
Given Corollary~\ref{c.DZ.bd} and \eqref{e.Z.hkk} for $i=N-\ell$, the same argument that leads to \eqref{e.decoupling.bd.Abd} gives
\begin{align}
	\label{e.decoupling'.Abd}
	A'_{K,\calI,\calJ} 
	\leq 
	\frac{1}{(N-j_{K})^{\frac{3}{2}\ind_{K\neq K_*}}}\,\frac{c^{|K|}}{ (\log\bb^{-1})^{(|K|-1)/2} }
	\Big( r^2 \log \frac{1+r^2}{r^2} \Big)^{\ind_{K= K_*}}\ .
\end{align}
Insert \eqref{e.decoupling'.Abd} into \eqref{e.decoupling'.expansion} and apply the argument after \eqref{e.decoupling.bd.expansion..} to bound the result.
Doing so shows that the last term in \eqref{e.decoupling'.3} is bounded by
\begin{align}
	\label{e.decoupling'.expansion.}
	r^2 \log \frac{1+r^2}{r^2}\Big(
		\sum_{m=2}^{N-\ell}
		\frac{c^m}{(\log\bb^{-1})^{(m-1)/2}}
		\
		\sum_{\vecu: \sum u_i=m}
		\
		\sum_{\calK}
		\
		\prod_{h=1}^{|\vecu|-1}
		\frac{1}{(N-j_{K_{h}})^{3/2}}
	\Big)\ ,
\end{align}
where the third sum runs over all $\calK\in\scrI'$ with $\type\calK=\vecu$, and $\type\calK$ was defined before \eqref{e.decoupling.bd.expansion..}.
The product in \eqref{e.decoupling'.expansion.} does not include $h=|\vecu|$ because the bound in \eqref{e.decoupling'.Abd} lacks the factor $1/(N-j_{K})^{3/2}$ when $K=K_*$ (and $K_*$ is the last interval when we order the intervals in $\calK\in\scrI'$).
To compensate for the lacked factor, we let $n:=m-|K_*|$, $\vecv:=(u_1,\ldots,u_{|\vecu|-1})$ and $\calL:=\calK\setminus\{K_*\}$ and write the expression within the parentheses in \eqref{e.decoupling'.expansion.} as
\begin{align}
	\label{e.decoupling'.expansion..}
	\sum_{n=0}^{N-\ell}
	\frac{c^{n}}{(\log\bb^{-1})^{n/2}}
	\
	\sum_{\vecv: \sum v_i=n}
	\
	\sum_{\calL}
	\
	\prod_{h=1}^{|\vecv|}
	\frac{1}{(N-j_{K_{h}})^{3/2}}
	\cdot
	\sum_{|K_*|=1}^{N-\ell}
	\frac{c^{|K_*|}}{(\log\bb^{-1})^{(|K_*|-1)/2}}
	\ ,
\end{align}
where the third sum runs over all $\calL$ such that $\type\calL=\vecv$ and that $\calL\cup\{K_*\}\in\scrI'$.
Bound the last sum by $c$ under the assumption that $\bb$ is small enough, forgo the condition $\calL\cup\{K_*\}\in\scrI'$ in the sum over $\calL$, and apply the argument after \eqref{e.decoupling.bd.expansion..} to conclude that $\eqref{e.decoupling'.expansion..}\leq c$.
This completes the proof.
\end{proof}

Let us now prove the result that will be used in Section~\ref{s.last.piece}.
Let $\logf(x_1,x_2):=\log(|x_1-x_2|^{-1}\vee 1)$ for $x_1,x_2\in\R^2$.
\begin{cor}\label{c.quenched}
There exists $c=c(c_0)$ such that, for $\bb\leq 1/c$, $\ell\in[1,N]$, and $L\geq 1$,
\begin{align}
	\label{e.quenched}
	\P\Big[
		\ind_{\Wevent}
		\frac{%
			\big( \hk(\e^2) \Ldot Z_{0,t_{N-\ell}} \big)^{\otimes 2} \Rdot (1+\logf) 
		}{%
			\big( \hk(\e^2) \Ldot Z_{0,t_{N-\ell}} \Rdot 1 \big)^2
		}
		> L + \log t^{-1}_{N-\ell}
	\Big]
	\leq
	\frac{c}{(\ell\log\bb^{-1})^{1/2}} +\frac{c}{L}\ .
\end{align}
\end{cor}

\begin{proof}
To begin, write
\begin{align}
	\logf(t_{N-\ell}^{-1/2}x) = \int_{0}^{1} \d r \, \frac{1}{r} \ind_{|x_1-x_2|\leq r\,t_{N-\ell}^{1/2}} \leq c \int_{0}^{1} \d r\, \frac{1}{r} \hkk(t_{N-\ell}r^2,x)\ ,
	\quad
	x=(x_1,x_2)\ ,
\end{align}
apply $(\hk(\e^2) \Ldot Z_{0,t_{N-\ell}})^{\otimes 2}$ to both sides of the inequality, multiply the result by $\ind_{\Wevent}$, divide the result by the denominator in \eqref{e.decoupling'}, take $\E$ on both sides of the result, move $\E$ into the integral over $r$, and apply Lemma~\ref{l.decoupling'}.
Doing so gives
\begin{align}
	\E\, \ind_{\Wevent}
	\frac{%
			\big( \hk(\e^2) \Ldot Z_{0,t_{N-\ell}} \big)^{\otimes 2} \Rdot \logf(t_{N-\ell}^{-1/2}\cdott)
	}{%
		(\prod_{i=1}^{N-\ell} \hk_{i-1} \Ldot Z_{i} \Rdot 1)^2
	}
	\leq
	c \int_0^1 \d r \, \frac{1}{r} \ r^2 \log \frac{1+r^2}{r^2}
	<
	\infty\ .
\end{align}
Combining this with Markov's inequality and Proposition~\ref{p.decoupling.bd} gives
\begin{align}
	\label{e.quenched.}
	\P\Big[
		\ind_{\Wevent}
		\frac{%
			\big( \hk(\e^2) \Ldot Z_{0,t_{N-\ell}} \big)^{\otimes 2} \Rdot \logf(t_{N-\ell}^{-1/2}\cdott)
		}{%
			\big( \hk(\e^2) \Ldot Z_{0,t_{N-\ell}} \Rdot 1 \big)^2
		}
		> L 
	\Big]
	\leq
	\frac{c}{(\ell\log\bb^{-1})^{1/2}} +\frac{c}{L}\ .
\end{align}
Next, observe that $1 + \log t^{-1/2} + \logf(t^{-1/2}\,\cdott) \geq 1+\logf$ and that $(\hk(\e^2) \Ldot Z_{0,t_{N-\ell}})^{\otimes 2} \Rdot f_0 = f_0 \,( \hk(\e^2) \Ldot Z_{0,t_{N-\ell}} \Rdot 1 )^2$ for any constant function $f_0$ of $x$.
Combining these observations for $t=t_{N-\ell}$ and $f_0=1 + \log t_{N-\ell}^{-1/2}$, we see that \eqref{e.quenched.} implies \eqref{e.quenched}.
\end{proof}

\section{Comparison for the last interval}
\label{s.last.piece}

Given Corollary~\ref{c.clt}, we complete the proof of Theorem~\ref{t.main} by proving the following comparison bounds.
The comparison bounds we will prove are the following:
For any $r_{\ell,\bb}$ such that $r_{\ell,\bb}/(\ell\log\bb^{-1})^{\alpha}\to\infty$ as $(\ell,\bb)\to(\infty, 0)$ for some $\alpha>3/2$,
\begin{align}
	\label{e.upbd}
	\P\big[ 
		\hk(\e^2) \Ldot Z_{0,1} \Rdot 1 
		>
		e^{r} \, \hk(\e^2) \Ldot Z_{0,t_{N-\ell}} \Rdot 1 
	\big]
	&\leq e^{-r}\ ,
	\quad
	r \geq 0\ ,
\\
	\label{e.lwbd}
	\lim_{(\ell,\bb)\to(\infty,0)}
	\sup_{\e\leq 1/2, |\theta|\leq c_0}
	\P\big[ 
		\hk(\e^2) \Ldot Z_{0,1} \Rdot 1 
		<
		e^{-r_{\ell,\bb}} \, \hk(\e^2) \Ldot Z_{0,t_{N-\ell}} \Rdot 1 
	\big]
	&= 0\ .
\end{align}
Indeed, these bounds together with Lemma~\ref{l.clt} and Corollary~\ref{c.clt} imply Theorem~\ref{t.main}.

To prove these bounds, we decompose $Z_{0,1}$ into $Z_{0,t_{N-\ell}}\Cdot Z_{t_{N-\ell},1}$.
Recall from \eqref{e.Cdot}--\eqref{e.Cdot.} that $\Cdot$ is defined through a limit.
We write
\begin{align}
	\label{e.lastpiece.decomp}
	\hk(\e^2) \Ldot Z_{0,1} \Rdot 1
	=
	\lim_{\delta\to 0} \Phi_{\delta} \Ldot Z_{t_{N-\ell},1} \Rdot 1\ ,
	&&
	\Phi_{\delta}
	=
	\Phi^{\theta}_{\ell,\delta}(x)
	:=
	\hk(\e^2) \Ldot Z^{\theta}_{0,t_{N-\ell}} \Rdot \hk(\delta^2,\cdott-x)\ .
\end{align}
For purely technical reasons, instead of $\hk(\e^2)\Ldot Z_{0,1}\Rdot 1$, below we will work with $\Phi_{\delta} \Ldot Z_{t_{N-\ell},1} \Rdot 1$ for $\delta>0$.
The bounds we obtain below are uniform in $\delta\leq 1$, which allows us to send $\delta\to 0$ in the end and recover statements about $\hk(\e^2) \Ldot Z_{0,1} \Rdot 1$.
Write $\EE:=\E[\cdott|\filt_{[0,t_{N-\ell}]}]$ and similarly for $\PP$.
Given the independence in Axiom~\ref{d.shf.inde}, we equivalently view $\EE$ and $\PP$ as the law of $Z_{s,t}$ for $s,t\in[t_{N-\ell},1]$.

To prove \eqref{e.upbd}, use Markov's inequality to write
\begin{align}
	\label{e.upbd.markov}
	\PP \big[ \Phi_{\delta} \Ldot Z_{t_{N-\ell},1} \Rdot 1 \leq  e^{r}\,\hk(\e^2) \Ldot Z_{0,t_{N-\ell}} \Rdot 1 \big]
	\geq
	1 - \frac{\EE\,\Phi_{\delta} \Ldot Z_{t_{N-\ell},1} \Rdot 1}{e^{r}\,\hk(\e^2) \Ldot Z_{0,t_{N-\ell}}\Rdot 1} \ .
\end{align}
In the last expectation, note that $\EE$ acts only on $Z_{t_{N-\ell},1}$, use $\EE\,Z_{t_{N-\ell},1}=\heatsg_{t_{N-\ell},1}$, write $\ip{\Phi_{\delta},\heatsg_{t_{N-\ell},1} 1}=\ip{\Phi_{\delta},1}$, and observe that the inner product evaluates to $\hk(\e^2)\Ldot Z_{0,t_{N-\ell}} \Rdot 1$.
Doing so shows that the fraction \eqref{e.upbd.markov} is equal to $e^{-r}$.
Now sending $\delta\to 0$ with the aid of \eqref{e.lastpiece.decomp} gives \eqref{e.upbd}.

To prove \eqref{e.lwbd}, we need to involve the conditional GMC structure from \cite{clark2025conditional}.
To begin, recall that the polymer measures associated with $Z$ were constructed in \cite{clark2024continuum}.
Let $\Pathsp:=\Csp([t_{N-\ell},1],\R^2)$ and write $\PM:=\PM^{\theta}_{[t_{N-\ell},1]}$ for the polymer measure over $[t_{N-\ell},1]$, which is a random measure on $\Pathsp$.
View $\Phi_{\delta}(x)$ as a function of $\ppath\in\Pathsp$ by substituting in $x=\ppath(t_{N-\ell})$ so that
\begin{align}
	\label{e.Z.M}
	\Phi_{\delta} \Ldot Z_{t_{N-\ell},1} \Rdot 1
	=
	\PM[\Phi_{\delta}]
	:=
	\int_{\Pathsp} \PM(\d\ppath) \, \Phi_{\delta}(\ppath(t_{N-\ell}))\ .
\end{align}
Next, decrease the coupling constant by $\aa>0$ and write $\PMm:=\PM^{\theta-\aa}_{[t_{N-\ell},1]}$ for the resulting polymer measure.
Let $\inters:\Pathsp^2\to[0,\infty]$ be the intersection local time constructed in \cite{clark2025planar}, see also \cite[Section~2.2]{clark2025conditional} for a quick review, and let $\noise=(\noise_{1},\noise_{2},\ldots)$ consist of iid standard Gaussians.
Given these data, we now consider the GMC $\gmc(\PMm,\sqrt{\aa}\noise)$ with base measure $\PMm$, covariance kernel $\inters$, and Gaussian input $\sqrt{\aa}\noise$, where we take $\PMm$ and $\sqrt{\aa}\noise$ to be independent.
We call $\gmc(\PMm,\sqrt{\aa}\noise)$ a conditional GMC because its base measure is random and the GMC is understood conditioned on the base measure.
For the ease of presentation, we defer the full definition of $\gmc(\PMm,\sqrt{\aa}\noise)$ to Appendix~\ref{s.a.gmc} and just mention that by \cite[Theorem~1.1]{clark2025conditional}
\begin{align}
	\label{e.cgmc}
	\PM = \gmc( \PMm, \sqrt{\aa}\noise) \ ,
	\qquad
	\PM = \PM^{\theta}_{[t_{N-\ell},1]}\ ,
	\
	\PMm = \PM^{\theta-\aa}_{[t_{N-\ell},1]}\ ,
\end{align}
where the first equality holds in law.
For the convenience of presentation, we suitably extend the probability space to include $\PMm$ and $\noise$ so that they are independent of $Z_{s,t}$ for $s,t\in[0,t_{N-\ell}]$ and that the first equality in \eqref{e.cgmc} holds samplewise.
Accordingly, write $\EE$ for the joint law of $\PMm$ and $\noise$.

In Appendix~\ref{s.a.gmc}, we apply the argument in \cite{morenoflores2014positivity} to the conditional GMC in \eqref{e.cgmc} to obtain the following bound.
Let 
\begin{align}
	R  := \frac{\PMm{}^{\otimes 2} [\Phi_{\delta}^{\otimes 2} e^{\aa\inters}] }{ \PMm[\Phi_{\delta}]^2}\ ,
	\qquad
	R' := \frac{ \PMm{}^{\otimes 2} [\Phi_{\delta}^{\otimes 2} \inters e^{\aa\inters} ] }{ \PMm[\Phi_{\delta}]^2}\ ,
\end{align}
where these variables depend on $\theta,\ell,\aa,\delta$ and $\bb,\e$ (through $t_{N-\ell},N$).
Then, $\P$-a.s. the following bound holds for all $r\geq 0$
\begin{align}
	\label{e.morenoflores}
	&\PP\Big[ 
		\PM[\Phi_{\delta}]
		< e^{-r} \, \PMm[\Phi_{\delta}]
	\Big]
	\leq
	2\,\EE\exp\Big( - \frac{1}{2^7\aa R R'}\big( r - \log 2 - \sqrt{ 2^7 \aa R R' \log 2^4R } \big)_+^2 \Big)\ .
\end{align}
To use \eqref{e.morenoflores}, we need to establish a few bounds.
Recall $\Wevent$ from \eqref{e.Wevent} and write $\P[ X<Y, \Wevent ]=\P[ \{X<Y\}\cap\Wevent ]$, etc.
\begin{lem}\label{l.quench.bds}
There exists $c=c(c_0)$ such that for all $\ell\in\Z_{\geq 1}$, $\bb\leq 1/c$, $\aa\geq 0$, $\delta\leq 1$, and $L\geq 1$, 
\begin{align}
	\label{e.quench.bd1}
	&\P\Big[ 
		\PMm[\Phi_{\delta}]
		< 
		\tfrac{1}{2} \ \hk(\e^2) \Ldot Z_{0,t_{N-\ell}} \Rdot 1\ , \Wevent
	\Big]
	\leq
	\frac{c\,\ell\log\bb^{-1}}{\aa+1} + \frac{c}{(\ell\log\bb^{-1})^{1/2}}\ ,
\\
	\label{e.quench.bd2}
	&\P\big[ R > L\ , \Wevent \big] + \P\big[ R' > L\ , \Wevent \big]
	\leq
	\frac{c\,\ell\log\bb^{-1}}{\aa+1} + \frac{c\,\ell\log\bb^{-1}}{L} + \frac{c}{(\ell\log\bb^{-1})^{1/2}}\ .
\end{align}
\end{lem}
\begin{proof}
Recalling that $\sg^{n,\theta'}$ and $\sgg^{n,\theta'}$ denote the kernels of the $n$th moments of $Z^{\theta'}$ and $W^{\theta'}$ respectively, we begin by establishing the following identity and bounds
\begin{align}
	\label{e.quench.bds.a}
	\EE\,\PMm[\Phi_{\delta}] &= \hk(\e^2) \Ldot Z_{0,t_{N-\ell}} \Rdot 1\ ,
\\
	\label{e.quench.bds.b}
	\Varr\,\PMm[\Phi_{\delta}] &\leq \big( \hk(\e^2) \Ldot Z_{0,t_{N-\ell}} \big)^{\otimes 2} \Rdot \, \sgg^{2,\theta-\aa}(1+\delta^2-t_{N-\ell})1\ ,
\\
	\label{e.quench.bds.c}
	\EE\,\PMm{}^{\otimes 2}[\Phi_{\delta}^{\otimes 2} e^{\aa\inters}]
	&\leq \big( \hk(\e^2) \Ldot Z_{0,t_{N-\ell}} \big)^{\otimes 2} \Rdot \, \sg^{2,\theta}(1+\delta^2-t_{N-\ell})1\ ,
\\
	\label{e.quench.bds.d}
	\EE\,\PMm{}^{\otimes 2}[\Phi_{\delta}^{\otimes 2} \inters e^{\aa\inters}]
	&\leq \big( \hk(\e^2) \Ldot Z_{0,t_{N-\ell}} \big)^{\otimes 2} \Rdot \, \sg^{2,\theta+1}(1+\delta^2-t_{N-\ell})1\ .
\end{align}
Write $Z^{\theta-\aa}_{t_{N-\ell},1}$ and $W^{\theta-\aa}_{t_{N-\ell},1}$ as $Z'_{t_{N-\ell},1}$ and $W'_{t_{N-\ell},1}$ respectively to simplify notation.
To prove \eqref{e.quench.bds.a}, write $\PMm[\Phi_{\delta}]=\Phi_{\delta}\Ldot Z'_{t_{N-\ell},1}\Rdot 1$ and take $\EE$ on both sides.
Doing so shows that $\EE\,\PMm[\Phi_{\delta}]=\ip{\Phi_{\delta},\heatsg_{t_{N-\ell},1} 1} = \ip{\Phi_{\delta},1}$.
From the definition of $\Phi_{\delta}$ in \eqref{e.lastpiece.decomp}, we see that $\ip{\Phi_{\delta},1}$ is equal to the right hand side of \eqref{e.quench.bds.a}.
To prove \eqref{e.quench.bds.b}, write $\PMm[\Phi_{\delta}]=\Phi_{\delta}\Ldot Z'_{t_{N-\ell},1}\Rdot 1$ and note that $\Varr\,\Phi_{\delta}\Ldot Z'_{t_{N-\ell},1}\Rdot 1=\EE\,(\Phi_{\delta}\Ldot W'_{t_{N-\ell},1}\Rdot 1)^2$, which evaluates to $\ip{\Phi_{\delta}^{\otimes 2}, \sgg^{2,\theta-\aa}(1-t_{N-\ell})1}$.
Referring to the definition of $\Phi_{\delta}$ in \eqref{e.lastpiece.decomp} and noting that
\begin{align}
	\big( \hk(\delta^2)^{\otimes 2}\sgg^{2,\theta-\aa}(1-t_{N-\ell}) \big)(x,x')
	\leq
	\sgg^{2,\theta-\aa}(1-t_{N-\ell}+\delta^2,x,x')
\end{align}
(see \eqref{e.sgg2}), we see that the last inner product is bounded by the right hand side of \eqref{e.quench.bds.b}.
To prove \eqref{e.quench.bds.c}, first note that $\PMm{}^{\otimes 2}[\Phi_{\delta}^{\otimes 2} e^{\aa\inters}]=\PM[\Phi_{\delta}]^{2}$.
This follows from \eqref{e.cgmc} and the second moment formula of GMC (see \eqref{e.gmc.2ndmom}).
Next, write $\PM[\Phi_{\delta}]^{2}=(\Phi_{\delta}\Ldot Z_{t_{N-\ell},1}\Rdot 1)^2=\Phi_{\delta}^{\otimes 2}\Ldot Z_{t_{N-\ell},1}^{\otimes 2}\Rdot 1$, take $\EE$ on both sides, and simplify the result similarly to the above.
The result gives \eqref{e.quench.bds.c}.
To prove \eqref{e.quench.bds.d}, note that $\inters e^{\aa\inters}\leq e^{(\aa+1)\inters}$ and use \eqref{e.quench.bds.c} for $\aa\mapsto\aa+1$.

Let us prove \eqref{e.quench.bd1}.
Recall that $\logf(x_1,x_2):=\log(|x_1-x_2|^{-1}\vee 1)$.
From \eqref{e.jfn.intbd}, \eqref{e.sgg2}, and $\sg^{2,\theta'}=\heatsg+\sgg^{2,\theta'}$, it is straightforward to check that
\begin{align}
	\label{e.sg.sgg2.bd}
	0 < \sgg^{2,\theta'}(t) 1 \leq \frac{c(c_0)}{\theta'_-+1} (1+\logf)\ ,
	\qquad
	0 < \sg^{2,\theta'}(t) 1 \leq c(c_0)\, (1+\logf)\ , 
\end{align}
for all $\theta'\leq c_0$ and $t\geq 1/2$.
Apply Chebyshev's inequality to the random variable $\PMm[\Phi_{\delta}]$ wrt $\EE$ with the aid of \eqref{e.quench.bds.a}--\eqref{e.quench.bds.b} and \eqref{e.sg.sgg2.bd}, take $\E$ of the result, and bound it using Corollary~\ref{c.quenched}.
Doing so gives
\begin{align}
	\label{e.l.quench.bds}
	&\P\Big[ 
		\PMm[\Phi_{\delta}]
		< 
		\tfrac{1}{2} \ \hk(\e^2) \Ldot Z_{0,t_{N-\ell}} \Rdot 1\ , \Wevent
	\Big]
	\leq
	c\,\frac{L+\log t_{N-\ell}^{-1}}{\aa+1} + \frac{c}{(\ell \log\bb^{-1})^{1/2}} +\frac{c}{L}\ ,
\end{align}
for $L\geq 1$.
From here, writing $\log t_{N-\ell}^{-1} = 2\log\bb^{-1}\cdot(\frac{\log\e^{-1}}{\log\bb^{-1}}-N+\ell) \leq c\,\ell \log\bb^{-1}$ and taking $L=\ell\log\bb^{-1}$ give \eqref{e.quench.bd1}.

Let us prove \eqref{e.quench.bd2}.
Apply Markov's inequality to the random variables $\PMm{}^{\otimes 2}[\Phi_{\delta}^{\otimes 2} e^{\aa\inters}]/(\hk(\e^2) \Ldot Z_{0,t_{N-\ell}} \Rdot 1)^2$ and $\PMm{}^{\otimes 2}[\Phi_{\delta}^{\otimes 2} \inters\,e^{\aa\inters}]/(\hk(\e^2) \Ldot Z_{0,t_{N-\ell}} \Rdot 1)^2$ wrt $\EE$ with the aid of \eqref{e.quench.bds.c}--\eqref{e.quench.bds.d} and \eqref{e.sg.sgg2.bd}, take $\E$ of the result, and bound it using Corollary~\ref{c.quenched} with $L=\ell\log\bb^{-1}$, using $\log t_{N-\ell}^{-1} \leq c\, \ell \log\bb^{-1}$.
Doing so gives
\begin{subequations}
\label{e.l.quench.bds.}
\begin{align}
	\P&\Big[ 
		\frac{ \PMm{}^{\otimes 2}[\Phi_{\delta}^{\otimes 2} e^{\aa\inters}] }{ (\hk(\e^2) \Ldot Z_{0,t_{N-\ell}} \Rdot 1)^2 }
		> \til{L}\ , \Wevent
	\Big]
	+
	\P\Big[ 
		\frac{ \PMm{}^{\otimes 2}[\Phi_{\delta}^{\otimes 2}\,\inters e^{\aa\inters}] }{ (\hk(\e^2) \Ldot Z_{0,t_{N-\ell}} \Rdot 1)^2 }
		> \til{L}\ , \Wevent
	\Big]
\\
	&\leq
	\frac{c\,\ell\log\bb^{-1}}{\til{L}} + \frac{c}{(\ell\log\bb^{-1})^{1/2}}\ ,
\end{align}
\end{subequations}
for all $\til{L}\geq 1$.
Combining \eqref{e.l.quench.bds.} and \eqref{e.quench.bd1} and renaming $\til{L}/4\mapsto L$ give \eqref{e.quench.bd2}.
\end{proof}

Let us prove \eqref{e.lwbd}.
Combining \eqref{e.morenoflores} and Lemma~\ref{l.quench.bds} gives
\begin{subequations}
\label{e.lwbd.}
\begin{align}
	&\P\Big[ 
		\PM[\Phi_{\delta}]
		< e^{-r} \, \hk(\e^2)\Ldot Z_{0,t_{N-\ell}} \Rdot 1\ , \Wevent
	\Big]
	\leq
	\frac{c\,\ell\log\bb^{-1}}{\aa+1}+\frac{c\ell\log\bb^{-1}}{L}+\frac{c}{(\ell\log\bb^{-1})^{1/2}}
\\
	&+
	c\,\exp\Big(-\frac{1}{c\aa L^2}(r-c-c\sqrt{\aa L^2\log cL})_+^2\Big)\ .
\end{align}
\end{subequations}
Given any $r_{\ell,\bb}$ with $r_{\ell,\bb}/(\ell\log\bb^{-1})^{\alpha}\to\infty$ for $\alpha>3/2$, letting $\aa_{\ell,\bb}=L_{\ell,\bb}:= (\ell\log\bb^{-1})^{(2\alpha+3)/6}
$ and substituting in $\aa=\aa_{\ell,\bb}$, $L=L_{\ell,\bb}$, and $r=r_{\ell,\bb}$ make the right hand side of \eqref{e.lwbd.} tend to $0$ as $(\ell,\bb)\to(\infty,0)$.
Combining this fact with \eqref{e.Wevent.prob} gives \eqref{e.lwbd}.

\appendix

\section{Proof of Corollary~\ref{c.main}}
\label{s.a.pf.c.main}
 
Note that $\frac{1}{(\log\e^{-1})^{\beta_\eps'}} 
				\leq Z_\e(y) 
				\leq \frac{1}{(\log\e^{-1})^{\beta_\eps}}$ is equivalent to
        \begin{align}\label{e.321}
 \Big(\frac{1+\alpha_\eps}{2}-\beta_\eps'\Big)\log \log \eps^{-1}\leq        H_\eps(y)+\frac{1+\alpha_\eps}{2}\log \log \eps^{-1} \leq  \Big(\frac{1+\alpha_\eps}{2}-\beta_\eps\Big)\log \log \eps^{-1}\ .
        \end{align}
			Take $\beta_\eps:=\tfrac{1+\alpha_\eps}{2}-(\log\log\e^{-1})^{-1/3}$ and $\beta_\eps':=\tfrac{1+\alpha_\eps}{2}+(\log\log\e^{-1})^{-1/3}$, and write \eqref{e.321} as
\begin{align}
	      H_\eps(y)+\frac{1+\alpha_\eps}{2}\log \log \eps^{-1}\in I_\eps:=[-(\log\log \eps^{-1})^{2/3},(\log\log \eps^{-1})^{2/3}]
\end{align}
			for the interval $I_\eps$. Then we have
      \begin{align}
      \E |D\cap D_{\e}^{\compl}| 
  =\int_D \d y\, \P[y\in D_{\e}^{\compl}]=\int_D \d y\, \big( 1-\P[H_\eps(y)+\tfrac{1+\alpha_\eps}{2}\log \log \eps^{-1}\in I_\eps]\big)\ .
      \end{align}
The proof is complete once we use the translation invariance of $H_\e$ to replace $H_\e(y)$ with $H_\e(0)$, and apply Theorem~\ref{t.main}.

\section{Moments}
\label{s.a.Wmom}
Recall from Axiom~\ref{d.shf.mome} and \eqref{e.Wmomformula} that $\sg=\sg^{n,\theta}$ and $\sgg=\sgg^{n,\theta}$ describe the $n$th moments of $Z^{\theta}$ and $W^{\theta}$ respectively.

Let us recall the explicit formulas of $\sg$ and $\sgg$ and their properties.
Write $[n]:=[1,n]\cap\Z$ to simplify notation.
By a pair $\alpha=ij$ we mean an unordered pair of $i<j\in [n]$ and view it as a set $\alpha=\{i,j\}$.
Let $\pair[n]$ denote the set of such pairs, and let $[n]_{\alpha}:=\{\textc\}\cup([n]\setminus\alpha)$, where $\textc$ means the center of mass index and is taken as an element from outside of $[n]$.
Recall that $\hk(t,x)=\exp(-|x|^2/2t)/2\pi t$ is the heat kernel on $\R^2$.
With slight abuse of notation, when $|\alpha\cap\alpha'|=1$ we write $y_{\alpha\setminus\alpha'}=y_{i}$ where $\alpha\setminus\alpha'=\{i\}$ and do similarly for $y_{\alpha'\setminus\alpha}$.
For $\alpha\neq\alpha'\in\pair[n]$, define the integral operators $\heatsg_{\alpha}(t)$, $\heatsg_{\alpha}(t)^*$, $\heatsg_{\alpha\alpha'}(t)$, $\Jop_{\alpha}(t)=\Jop^{\theta}_{\alpha}(t)$ through their kernels as
\begin{subequations}
\label{e.ops}
\begin{align}
	\label{e.incoming}
	\heatsg_{\alpha}(t,y,x)
	&:=
	\prod_{i\in\alpha} \hk(t,\yc-x_i) \cdot \prod_{i\in[n]\setminus\alpha} \hk(t,y_i-x_i)
	=:
	\heatsg_{\alpha}^*(t,x,y)
	\ ,
\\
	\label{e.Jop}
	\Jop^{\theta}_{\alpha}(t,y, y')
	&:=
	4\pi\,\jfn^{\theta}(t) \, \hk(\tfrac{t}{2},\yc-\yc') \, \prod_{i\in [n]\setminus\alpha} \hk(t,y_i-y'_i)\ ,
\\
	\label{e.swapping}
	\begin{split}
	\heatsg_{\alpha'\alpha}(t,y',y)
	&:=
	\prod_{i\in[n]\setminus(\alpha'\cup\alpha)} \hk(t,y'_i-y_i)
	\\
	&\cdot\begin{cases}
		\hk(t,y'_{\alpha\setminus\alpha'}-\yc) \hk(t,\yc'-y_{\alpha'\setminus\alpha}) \hk(t,\yc'-\yc)\ ,
		&
		|\alpha'\cap\alpha| = 1\ ,
	\\
		\prod_{i\in\alpha'} \hk(t,\yc'-y_i) \cdot 
		\prod_{i\in\alpha} \hk(t,y_i'-\yc)\ ,
		&
		\alpha'\cap\alpha = \emptyset\ ,
	\end{cases}
	\end{split}
\end{align}
\end{subequations}
where $x\in\R^{2[n]}$ and $y,y'\in\R^{2[n]_\alpha}$ in \eqref{e.incoming}--\eqref{e.Jop}, and $y'\in\R^{2[n]_{\alpha'}}$ and $y\in\R^{2 [n]_{\alpha}}$ in \eqref{e.swapping}.
Let
\begin{align}
	\label{e.dgm}
	 \dgm[n]
	&:=	
	\big\{ \vecalpha=(\alpha_k)_{k=1}^m\in\pair[n]^m \, \big| \, m\in\N, \alpha_{k}\neq\alpha_{k+1}, k=1,\ldots, m-1 \big\}\ .
\end{align}
This set indexes certain diagrams, and hence the abbreviation $\dgm$; see \cite[Section~2]{gu2021moments}.
Write $|\vecalpha|:=m$ for the length of $\vecalpha\in\dgm$.
For a function $f=f(u,u',\ldots)$  depending on finitely many nonnegative $u$s, write $\int_{\Sigma(t)}\d \vecu\, f = \int_{u+u'+\cdots=t} \d \vecu\, f$ for the convolution-like integral.
For $\vecalpha\in \dgm[n]$, define the operator
\begin{align}
	\label{e.sgsum}
	\sgsum^{n,\theta}_{\vecalpha}(t)
	&:=
	\int_{\Sigma(t)} \d \vecu \
	\heatsg_{\alpha_{1}}(u_{\frac{1}{2}})^* \,
	\prod_{k=1}^{|\vecalpha|-1} \Jop^{\theta}_{\alpha_{k}}(u_{k}) \, \heatsg_{\alpha_{k}\alpha_{k+1}}(u_{k+\frac{1}{2}}) \cdot
	\Jop^{\theta}_{\alpha_{|\vecalpha|}}(u_{|\vecalpha|}) \, \heatsg_{\alpha_{|\vecalpha|}}(u_{|\vecalpha|+\frac{1}{2}})\ .
\end{align}
Hereafter, products of operators are understood in the indexed order going from left to right, meaning that $\prod_{\ell=1}^{k}\Top_\ell := \Top_1 \Top_2 \cdots \Top_k$.
By \cite[Theorem~1.1]{gu2021moments}, 
\begin{align}
\label{e.sg}
	\sg(t) = \sg^{n,\theta}(t)
	=
	\heatsg(t)
	+
	\sum_{\vecalpha\in \dgm[n]} 
	\sgsum^{n,\theta}_{\vecalpha}(t)\ .
\end{align}
As for $\sgg$, by \cite[Equation~(2.20)]{tsai2024stochastic},
\begin{align}
	\label{e.sgg}
	\sgg(t) &= \sgg^{n,\theta}(t) = \sum_{\vecalpha\in \dgm_*[n]} \sgsum^{n,\theta}_{\vecalpha}(t)\ ,
\\
	\label{e.dgm*}
	\dgm_*[n]
	&:=	
	\big\{ \vecalpha\in\dgm[n] \, \big| \, \alpha_1\cup\cdots\cup\alpha_{|\vecalpha|} = \{1,\ldots,n\} \big\}\ .
\end{align}
Call $\ip{f,\Top g}$ the expression of \textbf{$\Top$ tested against $f,g$}.
By \cite[Section~2]{surendranath2024two}, for any $f:\R^{2n}\to\R$ that decays exponentially, the series in \eqref{e.sg}--\eqref{e.sgg} converge absolutely when tested against $f,1$.
It is straightforward to check that 
\begin{align}
	\label{e.scaling}
	\sgsum^{n,\theta'}_{\vecalpha}(t,x,x')
	=
	s^{n} \sgsum^{n,\theta'+\log s}_{\vecalpha}\big(st,\sqrt{s}x,\sqrt{s}x'\big)\ ,
	\qquad
	\theta'\in\R\ , \ s >0\ ,
\end{align}
and hence the same scaling property holds for $\sg$ and $\sgg$.

Below, we will derive a bound on 
$
	\E(\hk(r^2) \Ldot W^{ -L \log r^{-1} }_{0,1} \Rdot 1)^{n}\ ,
$
for $r\leq 1/2$ and $L\geq 1$, and use the scaling property \eqref{e.scaling} to infer a bound on $\E(\hk_{i-1}\Ldot W_{i} \Rdot 1)^{n}$.

We begin with a preliminary bound.
Let
\begin{align}
	\label{e.Wmom.AB}
	&\Aop^{\theta'}_{\alpha}
	:= 
	\int_{0}^{r^{-2}} \d t \, \Jop^{\theta'}_{\alpha}(t)\ ,
	&&
	\Bop^{1}_{\alpha'\alpha}
	:= 
	\int_{0}^{1} \d t \,  \heatsg_{\alpha'\alpha}(t)\ ,
	&&
	\Bop^{2}_{\alpha'\alpha}
	:= 
	\int_{1}^{r^{-2}} \d t \,  \heatsg_{\alpha'\alpha}(t)\ ,
\\
	\label{e.Wmom.C}
	&
	&&
	\Cop^{1}_{\alpha}
	:= 
	\int_{0}^{1} \d t \,  \heatsg_{\alpha}(t)^*\ ,
	&&
	\Cop^{2}_{\alpha}
	:= 
	\int_{1}^{r^{-2}} \d t \,  \heatsg_{\alpha}(t)^*\ ,
\end{align}
$\Bop_{\alpha'\alpha}:=\Bop^1_{\alpha'\alpha}+\Bop^2_{\alpha'\alpha}$, and $\Cop_{\alpha}:=\Cop^1_{\alpha}+\Cop^2_{\alpha}$.

\begin{lem}\label{l.Wmom.}
For $r\leq 1/2$, $L\geq 1$, and $\Aop_{\alpha}=\Aop_{\alpha}^{-(L+2)\log r^{-1}}$,
\begin{align}
	\label{e.Wmom.}
	\E ( \hk(r^2) \Ldot W^{ -L \log r^{-1} }_{0,1} \Rdot 1)^{n}
	\leq
	\sum_{\vecalpha\in\dgm_*[n]}
	\Ip{ 
		\hk(1)^{\otimes n} , 
		\Cop_{\alpha_{|\vecalpha|}} \Aop_{\alpha_{|\vecalpha|}} \Bop_{\alpha_{|\vecalpha|}\alpha_{|\vecalpha|-1}} \cdots \Bop_{\alpha_2\alpha_1} \Aop_{\alpha_1}\, 1
	}\ .
\end{align}
\end{lem}
\begin{proof}
Use \eqref{e.sgg} to express the left hand side of \eqref{e.Wmom.} as a sum of integrals of operators tested against $\hk(r^2)^{\otimes n},1$ and apply the scaling property \eqref{e.scaling} for $\theta'=-L\log r^{-1}$ and $s=r^{-2}$.
In the result, observe that $\heatsg_{\alpha_{|\vecalpha|}}(u_{|\vecalpha|+1/2})1=1$, effectively removing the variable $u_{|\vecalpha|+1/2}$ and turning the domain of integration to $\sum u_{i} \in [0,r^{-2}]$.
Note that the kernels in \eqref{e.ops} and $\jfn$ are positive, so the integrand is also positive.
Release the range of integration to $u_i\leq r^{-2}$ for all $i$.
Doing so gives \eqref{e.Wmom.}.
\end{proof}

To continue, we need to develop some terminology.
Take as $[n]$ or $[n]_{\alpha}$ the index set, and consider a partition $\Omega=\{\omega_1,\ldots,\omega_{|\Omega|}\}$ of the index set.
Call a function $f$ on $\R^{2\,\cup\Omega}$ \textbf{$\Omega$-translation invariant} if it is defined everywhere except where variables coincide and is invariant under any translation where the amount of translation is constant within each $(y_j)_{j\in\omega_i}$.
More precisely,
\begin{align}
	&
	f: 
	\big\{
		y\in\R^{2\,\cup\Omega}\, \big| \,y_i\neq y_j,\forall ij\in\pair(\omega)\,,\, \omega\in\Omega
	\big\} 
	\longrightarrow \R\ ,
\\
	&
	f\big(((y_{i}+a_{\omega})_{i\in\omega})_{\omega\in\Omega}\big)=f(y)\,,\qquad
	\vec{a}\in\R^{2\Omega}\ .
\end{align}
Write $\norm{\cdott}_{\eta}:=\norm{\cdott}_{\Lsp^2(\R^{2\eta})}$ to simplify notation.
Given the translation invariance, the norm $\norm{\cdott}_{\cup\Omega}$ generally evaluates to infinity, and the proper notion of norm should mod out the translation.
Taking the case $\Omega=\{\omega\}$ for example, where $\omega=[n]$ or $[n]_{\alpha}$, we define the modulo-$\Omega$ norm by freezing any variable 
$
	\norm{f}_{\omega/\{\omega\}}
	:=
	\norm{f}_{\omega\setminus\{i\}}(y_i)
	=
	\sup_{y_i\in\R^2}
	\norm{f}_{\omega\setminus\{i\}}(y_i)
$
where the norm does not depend on the choice of $i\in\omega$ or $y_i\in\R^2$, and is hence equal to the written supremum.
More generally, taking any $i_\omega\in\omega$ and $y_\omega\in\R^2$ for each $\omega\in\Omega$, we define the \textbf{modulo-$\Omega$ $L^2$ norm} of an $\Omega$-translation invariant $f$ as
\begin{align}
	\label{e.modulo.norm}
	\norm{f}_{\cup\Omega/\Omega}
	:=
	\norm{ f }_{\cup\Omega\setminus\{i_\omega\,|\,\omega\in\Omega\}} (\vec{y}\,)
	=
	\sup_{\vec{y}\in\R^{2\Omega}} 
	\norm{ f }_{\cup\Omega\setminus\{i_\omega\,|\,\omega\in\Omega\}} (\vec{y}\,)\ .
\end{align}

We next bound the operators in Lemma~\ref{l.Wmom.} in the language of translation invariant functions
\begin{lem}\label{l.Wmom.transl}
Given any $\alpha'\neq\alpha\in\pair[n]$, any partition $\Omega$ of $[n]_{\alpha}$, any $\Omega$-translation invariant $\psi\geq 0$, and any $\Gop$ listed below, there exists a partition $O$ and an $O$-translation invariant $f \geq 0$ such that
\begin{subequations}
\begin{align}
	&\Gop \psi
	\leq
	c(n) g(r,L) f\ ,
	\quad
	\norm{f}_{\cup O/O} \leq \norm{\psi}_{\cup\Omega/\Omega}\ ,
\\
	&g(r,L) := 
	\begin{cases}
		\frac{1}{L\log r^{-1}} & \Gop = \Aop_{\alpha} := \Aop^{-(L+2)\log r^{-1}}_{\alpha}\ , \\
		1 & \Gop = \Bop^{1}_{\alpha'\alpha} \text{ or } \Cop^{1}_{\alpha}\ , \\
		\log r^{-1} &\Gop = \Bop^{2}_{\alpha'\alpha} \text{ or } \Cop^{2}_{\alpha}\ ,
	\end{cases}
\end{align}
\end{subequations}
for all $r\leq 1/2$ and $L\geq 1$.
\end{lem}

\begin{proof}
Let us lay out the proof structure.
Write $c(n)=c$ throughout this proof.
The proof of the statements for $\Cop^{1}_{\alpha}$ and $\Cop^{2}_{\alpha}$ are similar to those for $\Bop^{1}_{\alpha'\alpha}$ and $\Bop^{2}_{\alpha'\alpha}$, so we consider only the latter and that for $\Aop_{\alpha}$.
In each case of $\Gop=\Aop_{\alpha}, \Bop^{1}_{\alpha'\alpha}, \Bop^{2}_{\alpha'\alpha}$, we will construct a suitable $O$ and a suitable integral operator $\Hop(t)$ with a positive kernel such that
\begin{align}
	\label{e.tranl.Gop}
	\Gop \psi_{\Omega} 
	\leq 
	\int_{s}^{s'} \d t\, \Hop(t) \psi_{\Omega}\ ,
	\quad
	(s,s') = 
	\begin{cases}
		(0,r^{-2})\ , & \Gop=\Aop_{\alpha}:=\Aop^{-(L+2)\log r^{-1}}_{\alpha}\ ,  \\ 
		(0,1)\ , & \Gop=\Bop^1_{\alpha'\alpha}\ ,  \\ 
		(1,r^{-2})\ , & \Gop=\Bop^2_{\alpha'\alpha}\ ,
	\end{cases}
\end{align}
and that $\Hop(t)$ maps $\Omega$-translation invariant functions to $O$-translation invariant functions.
Further, letting $\normopm{\cdott}$ denote the operator norm from $\norm{\cdott}_{\, \cup\Omega/\Omega}$ to $\norm{\cdott}_{\, \cup O/O}$, we will show that
\begin{align}
	\label{e.tranl.Hop}
	\NOrmopm{ \int_{s}^{s'}\d t\, \Hop(t) }
	\leq
	\begin{cases}
		c/(L\log r^{-1})\ , & \Gop=\Aop_{\alpha}:=\Aop^{-(L+2)\log r^{-1}}_{\alpha}\ ,  \\
		c\ , & \Gop=\Bop^1_{\alpha'\alpha}\ ,  \\ 
		c\log r^{-1}\ , & \Gop=\Bop^2_{\alpha'\alpha}\ .
	\end{cases}	
\end{align}
Once these are done, combining \eqref{e.tranl.Gop}--\eqref{e.tranl.Hop} gives the desired results.

\medskip
\textbf{The case $\Gop=\Aop_{\alpha}$.}\
Take $O=\Omega$, let $\Hop'(t) := \hk(t/2)\otimes \hk(t)^{\otimes [n]_{\alpha}}$, and take $\Hop(t) = \jfn(t)\Hop'(t)$.
By definition, both sides of the inequality in \eqref{e.tranl.Hop} are equal, and it is readily checked that $\Hop(t)$ maps $\Omega$-translation invariant functions to $\Omega$-translation invariant functions.
Pick $i_{\omega}$ arbitrarily as before \eqref{e.modulo.norm} and write $\Hop'(t)\psi$ as an integral.
In the integral, use the fact that the heat kernel integrates to $1$ for $i\in\{i_{\omega}\,|\,\omega\in\Omega\}$ and the fact that the heat kernel contracts the $\Lsp^2$ norm for $i\notin\{i_{\omega}\,|\,\omega\in\Omega\}$.
Doing so shows that $\norm{\Hop'(t)\psi}_{\cup O/O} \leq \norm{\psi}_{\cup \Omega/\Omega}$ and therefore $\normopm{\Hop'(t)}\leq 1$.
This together with \eqref{e.jfn.intbd} verifies \eqref{e.tranl.Hop}.

\medskip
\textbf{The case $\Gop=\Bop^{1}_{\alpha'\alpha}$.}\
Take $\alpha=12$ and $\alpha'=23$ or $34$ to simplify notation.

We start by constructing $O$ and $\Hop(t)$.
Simply take $\Hop(t)=\heatsg_{\alpha'\alpha}(t)$ so that both sides of the inequality in \eqref{e.tranl.Gop} are equal.
To construct $O$, view $[n]_{\alpha}$ and $[n]_{\alpha'}$ as disjoint sets of vertices.
Slightly abusing notation, we write elements in $[n]_{\alpha}$ as $\textc,3,4$ etc and elements in $[n]_{\alpha'}$ as $\textc',1',5'$ etc, using prime to signify that the sets are disjoint.
In \eqref{e.swapping}, view $\yc',y'_{1}$ etc as being indexed by $\textc',1'$ and form an edge between a vertex in $[n]_{\alpha}$ and a vertex in $[n]_{\alpha'}$ if the corresponding $y$ and $y'$ variables appear in the same heat kernel in \eqref{e.swapping}.
This gives a bipartite graph, which we call $G^1$; see Figure~\ref{f.graph}.
Next, in $G^1$ add an edge between vertices in $[n]_{\alpha}$ if they belong to the same $\omega\in\Omega$, and call the resulting graph $G^1_{\Omega}$.
The connected components of $G^1_{\Omega}$ induce a partition on $[n]_{\alpha'}$, and we let $O$ be that partition.
It is readily checked that $\Hop(t)=\heatsg_{\alpha'\alpha}(t)$ maps $\Omega$-translation invariant functions to $O$-translation invariant functions.

\begin{figure}
\includegraphics[width=\linewidth]{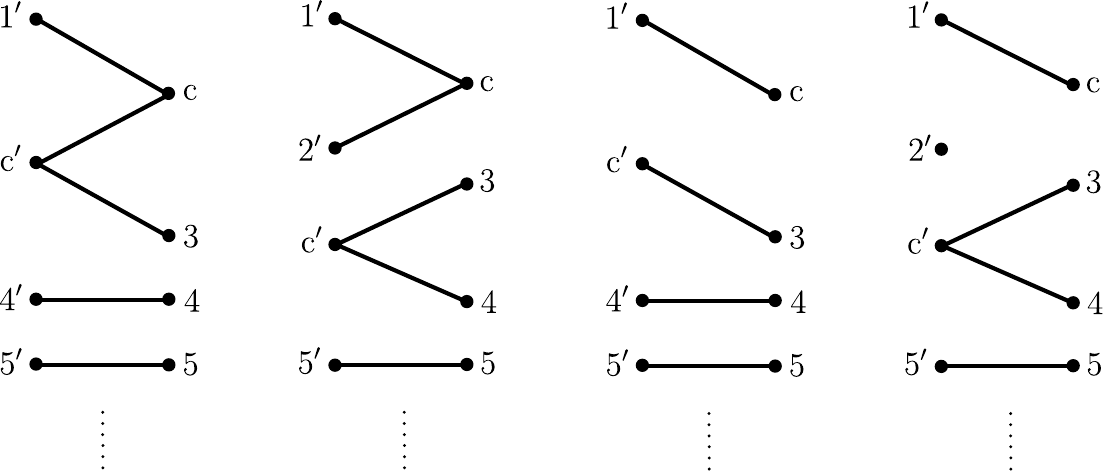}
\caption{From left to right, the graphs are $G^1$ for $\alpha=12$ and $\alpha=23$, $G^1$ for $\alpha=12$ and $\alpha=34$, $G^2$ for $\alpha=12$ and $\alpha=23$, and $G^2$ for $\alpha=12$ and $\alpha=34$.}
\label{f.graph}
\end{figure}

Next, let us specify our choice of $i_{\omega}$ and $i_{o}$ for $\omega\in\Omega$ and $o\in O$.
For $i\in[n]_{12}$, write $\omega_i$ for the $\omega\in\Omega$ that contains $i$; do similarly for $o_{i}$.
\begin{description}
\item[Choosing $i_{\omega}$ and $i_{o}$ when $\alpha=12$ and $\alpha'=23$]
Choose $i_{\omega_{\textc}}=\textc$, choose $i_{\omega_3}=3$ if $3\notin\omega_\textc$ (otherwise $i_{\omega_3}=i_{\omega_\textc}=\textc$), and choose $i_{o_{\textc'}}=i_{o_{1'}}=\textc'$.
The rest $\omega$s and $o$s are in bijection.
More precisely, writing $\pi:4\mapsto 4', 5\mapsto 5',\ldots$, we have $O\setminus\{o_{\textc'},o_{1'}\} = \{ \pi\omega\,| \omega\in\Omega\setminus\{\omega_{\textc},\omega_{3}\} \}$.
Choose $i_{\omega}$ arbitrarily for those $\omega$ in the last set, and choose $i_{o}=i_{\omega}$ for $o=\pi\omega$.
\item[Choosing $i_{\omega}$ and $i_{o}$ when $\alpha=12$ and $\alpha'=34$]
Choose $i_{\omega_{\textc}}=\textc$, choose $i_{\omega_3}=3$ if $3\notin\omega_\textc$ (otherwise $i_{\omega_3}=i_{\omega_\textc}=\textc$), choose $i_{\omega_4}=4$ if $4\notin\omega_\textc\cup\omega_3$ (otherwise similarly), choose $i_{o_{\textc'}}=\textc'$, and choose $i_{o_{1'}}=i_{o_{2'}}=1'$ if $1',2'\notin o_{\textc'}$ (otherwise similarly).
The rest $\omega$s and $o$s are in bijection just like before and we choose them similarly.
\end{description}

Let us verify \eqref{e.tranl.Hop}.
Using the readily verified inequality (see the proof of \cite[Equation~(2.15)]{surendranath2024two} for example)
$
	\hk(t,y'-y)e^{-a|y'|} \leq 2 e^{a^2t} \hk(2t,y'-y) e^{-a|y|}
$,
$a>0$,
in the kernel of $\heatsg_{\alpha'\alpha}(t)$ (in \eqref{e.swapping}) gives
\begin{align}
	\label{e.swapping.bd}
	\prod_{o\in O} e^{-|y'_{i_{o}}|} 
	\cdot 
	\heatsg_{\alpha'\alpha}(t,y',y)
	\leq
	c\, \heatsg_{\alpha'\alpha}(2t,y',y)
	\prod_{\omega\in\Omega} e^{-|y_{i_{\omega}}|/2}\ ,
	\quad
	t\in(0,1]\ .
\end{align}
Next, by definition for any $\Omega$-translation invariant $f$ and for any $g\in\Lsp^2(\R^2)$,
\begin{align}
	\label{e.modulo.norm.}
	\NOrm{
		f(y)\,\prod_{\omega\in\Omega} g(y_{i_\omega})
	}_{\Lsp^2(\R^{2\,\cup\Omega})}
	=
	\norm{f}_{\cup\Omega/\Omega}
	\prod_{\omega\in\Omega}\norm{g}_{\Lsp^2(\R^2)}
	\ .
\end{align}
Using this for $\Omega=\Omega$ and $g(y)=e^{-|y|/2}$ and for $\Omega\mapsto O$ and $g(y')=e^{-|y'|}$ in \eqref{e.swapping.bd} shows that $\normopm{\int_0^1 \d t \heatsg_{\alpha'\alpha}(t)}$ is bounded by a constant multiple of the $\Lsp^2(\R^{2[n]_{\alpha}})\to \Lsp^2(\R^{2[n]_{\alpha'}})$ operator norm of $\int_0^1 \d t\, \heatsg_{\alpha'\alpha}(2t)$; the latter is bounded by $c$ by \cite[Lemma~8.2(b)]{gu2021moments}.

\medskip
\textbf{The case $\Gop=\Bop^{2}_{\alpha'\alpha}$.}\
Take $\alpha=12$ and $\alpha'=23$ or $34$ to simplify notation.

We start by constructing $\Hop(t)$ and $O$.
In the kernel of $\heatsg_{\alpha'\,\alpha}(t)$ in \eqref{e.swapping}, bound $\hk(t,\yc'-\yc)\leq 1/2\pi t$ when $\alpha'=23$ and bound $\hk(t,y'_2-\yc)\leq 1/2\pi t$ when $\alpha'=34$.
We have
\begin{align}
	\label{e.Bop.bd}
	\Bop^{2}_{\alpha'\alpha}\, \psi
	&\leq
	\int_1^{r^{-2}} \frac{\d t}{2\pi t} \Hop^2(t)\psi\ ,
\\
	\label{e.Hop'}
	\Hop^2(t,y',y) 
	&:=
	\begin{cases}
		\hk(t, y'_1-\yc) \hk(t,\yc'-y_3) \prod_{i>3} \hk(t,y'_i-y_i)\ , & \alpha'= 23\ , \\
		\hk(t, y'_1-\yc) \prod_{j=3,4}\hk(t,\yc'-y_j) \cdot \prod_{i>4} \hk(t,y'_i-y_i)\ , & \alpha'= 34\ .
	\end{cases}
\end{align}
Let $\Hop(t):=\frac{1}{2\pi t}\Hop^2(t)$ so that \eqref{e.tranl.Gop} holds.
Next, construct the graph $G^2$ similarly to how we constructed $G^1$ but using the heat kernels in \eqref{e.Hop'} as the reference instead of \eqref{e.swapping}.
Since $\Hop^2$ arises by the bounding procedure before \eqref{e.Bop.bd}, the graph $G^2$ is the same as $G^1$ with the edge between $\textc'$ and $\textc$ removed when $\alpha'=23$ and the edge between $2'$ and $\textc$ removed when $\alpha'=34$; see Figure~\ref{f.graph}.
Next, in $G^{2}$, make an edge between the vertices in $[n]_{\alpha}$ if they belong to the same $\omega\in\Omega$, and let $G^{2}_{\Omega}$ denote the resulting graph.
The connected components of $G^{2}_{\Omega}$ induce a partition on $[n]_{\alpha'}$, and we let $O$ be that partition.
It is readily checked that $\Hop^2(t)$ and hence $\Hop(t)$ map $\Omega$-translation invariant functions to $O$-translation invariant functions. 

We now verify \eqref{e.tranl.Hop}.
Choose $i_{\omega}$ and $i_{o}$ the same way as in the previous case, except that here we make $i_{o_{2'}}=2'$ when $\alpha'=34$:
Unlike in the previous case, here $o_{2'}=\{2'\}$.
Write $\Hop^2(t)\psi$ as an integral with the aid of \eqref{e.Hop'}.
\begin{description}
\item[When $\alpha=12$ and $\alpha'=23$, or $\alpha=12$ and $\alpha'=34$ and $i_{\omega_3}\neq i_{\omega_4}$]
In the integral, use the fact that the heat kernel integrates to $1$ for $i\in\{i_{\omega}\,|\,\omega\in\Omega\}$ and the fact that the heat kernel contracts the $L^2$ norm for $i\notin\{i_{\omega}\,|\,\omega\in\Omega\}$.
Doing so shows that $\norm{\Hop^2(t)\psi}_{\cup O/O}\leq \norm{\psi}_{\cup \Omega/\Omega}$ and therefore $\normopm{\Hop^2(t)}\leq 1$.
\item[When $\alpha=12$, $\alpha'=34$, and $i_{\omega_3}= i_{\omega_4}$]
In the integral, use the fact that the heat kernel integrates to $1$ for $i\in\{i_{\omega}\,|\,\omega\in\Omega\}\setminus\{4\}$ and the fact that the heat kernel contracts the $L^2$ norm for $i\notin\{i_{\omega}\,|\,\omega\in\Omega\}$, and apply the Cauchy--Schwarz inequality over $y_4$ with the aid of $\int_{\R^2} \d y_4 \,\hk(t,\yc'-y_4)^2 = 1/4\pi t$.
Doing so gives $\norm{\Hop^2(t)\psi}_{\cup O/O}\leq \norm{\psi}_{\cup \Omega/\Omega}(4\pi t)^{-1/2}$, so $\normopm{\Hop^2(t)}\leq (4\pi t)^{-1/2}\leq 1$, since here $t\geq 1$.
\end{description}
Combining $\normopm{\Hop^2(t)}\leq 1$ with $\Hop(t):=\frac{1}{2\pi t}\Hop^2(t)$ verifies \eqref{e.tranl.Hop}.
\end{proof}

We are now ready to prove the needed moment bounds on $W$.
Recall that $W^{\theta}_i=W^{\theta}_{t_{i-1},t_i}$ and $t_i$ depends on $\bb$.
\begin{cor}\label{c.Wmom}
For every $n\in \N$, there exist $c_1=c_1(n)$ and $c_2=c_2(n,c_0)$ such that 
\begin{align}
	\label{e.Wmom}
	\E ( \hk(r^2) \Ldot W^{ -L \log r^{-1} }_{0,1} \Rdot 1)^{n}
	&\leq
	c_1\, L^{-\lceil n/2\rceil}\ ,
\\
	\label{e.Wi.mom}
	\E ( \hk_{i-1} \Ldot W^{\theta}_{i} \Rdot 1)^{n} &\leq c_1\, (N-i)^{-\lceil n/2\rceil}\ ,
\end{align}
for all $L\geq c$, $r\in(0,1/2]$, $i\in[1,N-c]$, $\bb\leq 1/2$, and $\theta\leq c_0$.
\end{cor}

\begin{proof}
To prove \eqref{e.Wmom}, combine the latter two results in Lemma~\ref{l.Wmom.transl} using $\Bop_{\alpha'\alpha}=\Bop^1_{\alpha'\alpha}+\Bop^2_{\alpha'\alpha}$ and $\Cop_{\alpha}=\Cop^1_{\alpha}+\Cop^2_{\alpha}$ to obtain that, for every $\Omega$ and $\psi$ as in Lemma~\ref{l.Wmom.transl},
\begin{align}
	\label{e.Bop.Cop.bd}
	\Bop_{\alpha'\alpha} \psi \leq c(n)\log r^{-1} \sum_{O} f_O\ ,
	\quad
	\Cop_{\alpha} \psi \leq c(n) \sum_{O} g_O\ ,	 
\end{align}
for some nonnegative $O$-translation invariant $f_O$ and $g_O$ with $\norm{f_O}_{\cup O/O} ,\norm{g_O}_{\cup O/O} \leq \norm{\psi}_{\cup \Omega/\Omega}$, where the first sum runs over all partitions of $[n]_{\alpha'}$ and the second sum runs over all partitions of $[n]$.
Use \eqref{e.Bop.Cop.bd} and the first result in Lemma~\ref{l.Wmom.transl} in \eqref{e.Wmom.}, note that $1$ is $\Omega_1$-translation invariant for $\Omega_1=\{\{i\}\,|\,i\in[n]_{\alpha_1}\}$ and $\norm{1}_{[n]_{\alpha_1}/\Omega_1}=1$, and use the property that $\ip{\hk(1)^{\otimes n},f}\leq (4\pi)^{(n-|O|)/2}\norm{f}_{[n]/O}$.
Doing so shows that the inner product is bounded by $(c(n)/L)^{|\vecalpha|}$.
Observe that there are at most $\binom{n}{2}^{m}$ many $\vecalpha$s with $|\vecalpha|=m$ and that every $\vecalpha\in\dgm_*[n]$ satisfies $|\vecalpha|\geq\lceil n/2\rceil$.
Bounding the sum using these observations gives \eqref{e.Wmom}.

The bound~\eqref{e.Wi.mom} follows from \eqref{e.Wmom}.
To see why, use the analog of \eqref{e.scaling} for $W$, set $s=t_{i}-t_{i-1}$, take $r=(t_{i-1}/t_i)^{1/2}$ and $L=(-\theta+\log s^{-1})/\log r^{-1}$, and bound $\log r^{-1}\leq c\log\bb^{-1}$ and $\log s^{-1} \geq 2\log \bb^{-1}\cdot (N-i) - c $ with the aid of \eqref{e.ti} and \eqref{e.logti}.
\end{proof}

Let us now prove Lemma~\ref{l.clt}.
\begin{proof}[Proof of Lemma~\ref{l.clt}]
Recall $\Wevent$ and $\Wevent_i$ from \eqref{e.Wevent} and let 
\begin{align}
	X_i:=  Y_i - \E\, Y_i\ , \qquad  
	Y_i := \ind_{\Wevent_i} \log \hk_{i-1}\Ldot Z_{i}\Rdot 1\ .
\end{align}
We seek to apply Lindeberg's CLT to $Y_1+\ldots+Y_{N-\ell}$.
To this end, write $\hk_{i-1}\Ldot Z_{i}\Rdot 1=1+\hk_{i-1}\Ldot W_{i}\Rdot 1$, Taylor expand the $\log$ accordingly, noting that $\Wevent_i$ implies $\hk_{i-1}\Ldot W_{i}\Rdot 1 \geq -1/2$.
Doing so gives
\begin{align}
    \label{e.clt.Yi}
	Y_i = \ind_{\Wevent_i} \, \hk_{i-1}\Ldot W_{i}\Rdot 1 - \ind_{\Wevent_i} \,\tfrac{1}{2}\, (\hk_{i-1}\Ldot W_{i}\Rdot 1)^2 + U_i\ ,
	\qquad
	|U_i| \leq c\, |\hk_{i-1}\Ldot W_{i}\Rdot 1|^3\ .
\end{align}
Adopt the big O notation.
Given \eqref{e.clt.Yi}, using \eqref{e.Wi.mom} and $\P[\Wevent^{\compl}_i] \leq c(k)/(N-i)^{k}$ for $i\geq c(k)$ and every $k\geq 1$ (which follows from \eqref{e.Wi.mom}) yields
\begin{align}
	\E\, Y_i = -\tfrac{1}{2} \E\, (\hk_{i-1}\Ldot W_{i}\Rdot 1)^2 + O(\tfrac{1}{(N-i)^{3/2}})\ ,
	\ \ 
	\Var\, Y_i = \E\, (\hk_{i-1}\Ldot W_{i}\Rdot 1)^2 + O(\tfrac{1}{(N-i)^2})\ ,
\end{align}
and $\E\,(Y_i-\E\, Y_i)^4 \leq c/(N-i)^2$.
Further, by \eqref{e.W.2ndmom} and the definition of $N$ in \eqref{e.ti},
\begin{align}
    \sum_{i=1}^{N-\ell}\E\, (\hk_{i-1}\Ldot W_{i}\Rdot 1)^2
    =
    (1+f(\e,\ell,\bb))\log\log\e^{-1} \ ,
\end{align}
for some $f$ such that $\lim_{(\ell,\bb)\to(\infty, 0)}\lim\sup_{\e\to 0} |f|= 0$. 
Given these properties, applying Lindeberg's CLT for fixed $\ell,\bb^{-1}$ and later sending $(\ell,\bb)\to(\infty,0)$ show that \eqref{e.clt} holds for $Y_1+\ldots+Y_{N-\ell}$ in place of $H'_{\e,\ell,\bb}$.
Since these two random variables are the same on $\Wevent$ and since the probability of $\Wevent$ tends to $1$ by \eqref{e.Wevent.prob}, the desired result \eqref{e.clt} follows.
\end{proof}

\section{Conditional GMC}
\label{s.a.gmc}
We begin by recalling the GMC formalism used in \cite{clark2025conditional}, which mostly draws on \cite{shamov2016gaussian}, and we remark that the theory of GMC was first developed by \cite{kahane1985} and point to \cite{rhodes2014gaussian} for a review. 
Take a Polish space $\pspace$, take a positive locally finite (finite on bounded sets) Borel measure $\base$ on $\pspace$ as the base measure, write $\Cloc(\pspace)$ for the space of bounded continuous functions with bounded support, and let $\noise=(\noise_1,\noise_2,\ldots)$ consist of iid Gaussians.
Next, take as the covariance kernel a symmetric $\kernel:\pspace^2\to[0,\infty]$.
Put $\base_g(\d x):=\base(\d x)\,g(x)$ and further assume that there exists $g_0\in\Cloc(\pspace)$ such that $\inf_B g_0>0$ for all bounded $B$ and that $\kernel$ defines a positive Hilbert--Schmidt integral operator on $\Lsp^2(\pspace,\base_{g_0})$.
Given these data, consider the subcritical GMC with base measure $\base$, covariance kernel $\kernel$, and Gaussian input $\sqrt{\aa}\noise$.
If exists, it is a random (in $\noise$) Borel measure on $\pspace$, and we write it as $\gmc(\base,\sqrt{\aa}\noise)=\gmc(\base,\sqrt{\aa}\noise)(\d x)$.
It is common to have unit variance for the Gaussian input and put the $\aa$ dependence in the covariance kernel, but we do it the other way because it is more convenient in the context of polymer measures.
Write $\EE$ for the law of $\noise$.

Under the current assumptions, if it exists, the GMC permits a more explicit expression as the limit of a Doob martingale.
Take any $\varphi:\pspace\to[0,\infty]$ with
\begin{align}
	\label{e.gmc.exp}
	\base^{\otimes 2}_{\varphi}\big[e^{\aa'\kernel}\big]
	:=
	\int_{\pspace^2} \base_{\varphi}^{\otimes 2}\,e^{\aa'\kernel}
	:=
	\int_{\pspace^2} \base^{\otimes 2}\, \varphi^{\otimes 2} e^{\aa'\kernel} <\infty\ ,
	\qquad
	\aa'\in\R\ .
\end{align}
Since $\kernel$ defines a positive operator on $\Lsp^2(\pspace,\base_{g_0})$ and $\inf_{B} g_0>0$ for all bounded $B$, $\kernel$ also defines a positive operator on $\Lsp^2(\pspace,\base_{\varphi})$, which is Hilbert--Schmidt by \eqref{e.gmc.exp}.
Enumerate its (strictly) positive eigenvalues and the corresponding eigenvectors as $\lambda_1,\lambda_2,\ldots$ and $u_1,u_2,\ldots$, and let
\begin{align}
	\label{e.gmck}
	\gmc_n(\base,\sqrt{\aa}\noise)(\d x)
	&:=
	\base(\d x) \, \exp \sum_{i=1}^{n} \Big(
		\sqrt{\aa\lambda_i}\noise_i u_{i}(x) - \frac{1}{2} \aa\lambda_{i} u_i(x)^2
	\Big)\ .
\end{align}
From the randomized shift characterization of GMCs of \cite{shamov2016gaussian}, it is not hard to check that
\begin{align}
	\label{e.gmc.doob}
	\EE\big[\, \gmc(\base,\sqrt{\aa}\noise)[\varphi]\,\big|\,\noise_1,\ldots,\noise_n\big]
	=\gmc_n(\base,\sqrt{\aa}\noise)[\varphi]\ .
\end{align}
Namely, the expression in \eqref{e.gmc.doob} forms a Doob martingale, so in particular
\begin{align}
	\label{e.gmc.mg}
	\gmc(\base,\sqrt{\aa}\noise)[\varphi]
	&= \lim_{n\to\infty} \gmc_n(\base,\sqrt{\aa}\noise)[\varphi]\ ,
	\qquad
	\text{a.s.\ and in }L^1(\PP)\ .
\end{align}
The Doob martingale is also useful for deriving bounds.
Letting $\kernel_n(x,y):=\sum_{i=1}^{n}\lambda_i u_i(x)u_i(y)$ and combining \eqref{e.gmck}--\eqref{e.gmc.doob} with Jensen's inequality give
\begin{align}
	\label{e.gmc.mombd}
	\base_{\varphi}^{\otimes 2} \big[  e^{\aa\kernel_n} \big]
	=
	\EE\big( \gmc_n(\base,\sqrt{\aa}\noise)[\varphi] \big)^2
	\leq
	\EE\big( \gmc(\base,\sqrt{\aa}\noise)[\varphi] \big)^2\ .
\end{align}
Sending $n\to\infty$ recovers the standard second moment formula
\begin{align}
	\label{e.gmc.2ndmom}
	\EE\big( \gmc(\base,\sqrt{\aa}\noise)[\varphi] \big)^2
	= \base_{\varphi}^{\otimes 2}\big[ e^{\aa\kernel} \big]
	:= \base^{\otimes 2}\big[ \varphi^{\otimes 2}\, e^{\aa\kernel} \big]\ .
\end{align}

Let us explain how the above setting can be specialized to the one relevant to this paper.
To this end, take $\pspace=\Pathsp:=\Csp([t_{N-\ell},1];\R^2)$, equip it with the supremum norm, take $\base=\PMm:=\PM^{\theta-\aa}_{[t_{N-\ell},1]}$, and take $\kernel=\inters$, the intersection local time constructed in \cite{clark2025planar}; see also \cite[Section~2.2]{clark2025conditional} for a quick review.
There are many choices for $g_0$, and we take $g_0(\ppath):=\exp(-|\ppath(t_{N-\ell})|)$ for the sake of concreteness.
The base measure $\PMm$ is random, and we take it independently of $\noise$.
Under the current specialization, our assumptions on $\kernel=\inters$ and $g_0$ hold because of Lemma~3.3, Equation~(3.37), and Lemma~A.3 of \cite{clark2025conditional}.
The existence of $\gmc(\PMm,\sqrt{\aa}\noise)$ follows from \cite[Lemma~3.3]{clark2025conditional} and \cite[Theorem~3]{shamov2016gaussian}.
We will further specialize $\varphi$ to the $\Phi_{\delta}$ defined in \eqref{e.lastpiece.decomp}.
Using \eqref{e.quench.bds.c} for $\aa\mapsto\aa'$ and $\theta\mapsto\theta-\aa+\aa'$ shows that \eqref{e.gmc.exp} holds for this specialization.

We now apply the argument of \cite{morenoflores2014positivity} to prove the following lower tail bound.
\begin{lem}\label{l.morenoflores}
Setup as above. 
Letting
$
	\rho := \base_{\varphi}^{\otimes 2} [e^{\aa\kernel}] / \base[\varphi]^2
$
and
$
	\rho' := \base_{\varphi}^{\otimes 2} [\kernel e^{\aa\kernel}] / \base[\varphi]^2
$,
we have that
\begin{align}
	\label{e.morenoflores.}
	\PP\Big[ \,\gmc(\base,\sqrt{\aa}\noise)[\varphi] < e^{-r} \base[\varphi] \,\Big]
	\leq
	2\exp\Big( - \frac{1}{2^7\aa\rho\rho'}\big( r - \log 2 - \sqrt{2^7\aa\rho\rho'\log 2^4\rho}\big)_+^2 \Big)\ .
\end{align}
\end{lem}
\noindent%
This lemma implies \eqref{e.morenoflores} by specializing the setup as described above.
\begin{proof}
To simplify notation, we take $\aa=1$ and write $\gmc(\base,\noise)=\gmc(\noise)$ and $\gmc_n(\base,\noise)=\gmc_n(\noise)$.

Let us explain how the desired result follows from its approximate version.
Recall $\kernel_n$ from before \eqref{e.gmc.mombd} and let $\rho_n$ and $\rho'_n$ be obtained by replacing the $\kernel$ in $\rho$ and $\rho'$ by $\kernel_n$.
We seek to prove that
\begin{align}
	\label{e.morenoflores.n}
	\PP\Big[ \,\gmc_n(\noise)[\varphi] < e^{-r} \base[\varphi] \,\Big]
	\leq
	2\exp\Big( - \frac{1}{2^7\rho_n\rho'_n}\big( r - \log 2 - \sqrt{2^7\rho_n\rho'_n\log 2^4\rho_n}\big)_+^2 \Big)\ .
\end{align}
Indeed, sending $n\to\infty$ turns the left hand side of \eqref{e.morenoflores.n} to that of \eqref{e.morenoflores.}.
The right hand side also converges, which amounts to saying that $\base_{\varphi}^{\otimes 2} [e^{\aa\kernel_n}]\to \base_{\varphi}^{\otimes 2} [e^{\aa\kernel}]$ and $\base_{\varphi}^{\otimes 2} [\kernel_n e^{\kernel_n}]\to \base_{\varphi}^{\otimes 2} [\kernel e^{\kernel}]$.
To see why, first note that $\kernel_n\to\kernel$ in measure wrt $\base_{\varphi}^{\otimes 2}$, so the analogous convergence holds for $e^{\kernel_n}$ and $\kernel_ne^{\kernel_n}$.
To upgrade this to convergence in $\Lsp^1(\pspace^2,\base_{\varphi}^{\otimes 2})$, combine \eqref{e.gmc.mombd}--\eqref{e.gmc.2ndmom} for $\aa\mapsto 2\aa$ to deduce that $\sup_n \base_{\varphi}^{\otimes 2}|e^{\kernel_n}|^2<\infty$, and apply the Cauchy--Schwarz inequality to the difference $\kernel_n e^{\kernel_n} - \kernel e^\kernel$ with the aid of $\base_{\varphi}^{\otimes 2}[(\kernel_n-\kernel)^2]=\sum_{i=n+1}^\infty \lambda_i^2 \to 0$.

To prove \eqref{e.morenoflores.n}, since $\gmc_n(\noise)$ only depends on the first $n$ entries of $\noise$, slightly abusing notation, we write $\noise=(\noise_1,\dots,\noise_n)$ for the rest of the proof.

As the first step toward proving \eqref{e.morenoflores.n}, we derive a useful lower bound.
Write $\overline{\gmc}_{n,\varphi}(\noise)(\d x):=\gmc_n(\noise)(\d x)\,\varphi(x)/\gmc_n(\noise)[\varphi]$ for the normalized measure, let
\begin{align}
	\label{e.morenoflores.A}
	A_1 := \big\{ \noise\in\R^n \, \big| \, \gmc_n(\noise)[\varphi] \geq \tfrac{1}{2} \base[\varphi] \big\}\ ,
	&&
	A_2(\alpha) := \big\{ \noise\in\R^n \, \big| \, 
		\overline{\gmc}_{n,\varphi}(\noise)^{\otimes 2}[\kernel_n] \leq \alpha
	\big\}\ ,
\end{align}
and $A(\alpha) := A_1 \cap A_2(\alpha)$, express
$\gmc_n(\noise)[\varphi]$ as follows
\begin{align}
	\gmc_n(\noise)[\varphi]
	=
	\overline{\gmc}_{n,\varphi}(\noise')\Big[ 
		\exp\sum_{i=1}^n (\noise_i-\noise'_i)\sqrt{\lambda_i} u_i
	\Big]
	\cdot \gmc_n(\noise')[\varphi]\ ,
\end{align}
apply Jensen's inequality wrt $\overline{\gmc}_{n,\varphi}(\noise')$, and apply the Cauchy--Schwarz inequality to the sum in the result.
Write $|\cdott|$ for the Euclidean norm on $\R^n$.
The preceding procedures give
\begin{align}
	\gmc_n(\noise)[\varphi]
	\geq
	\exp\Big( 
		-|\noise-\noise'| 
		\Big(\sum_{i=1}^n \lambda_i \,\big( \overline{\gmc}_{n,\varphi}(\noise')[u_i] \big)^2 \Big)^{1/2} 
	\Big)
	\cdot \gmc_n(\noise')[\varphi]\ .
\end{align}
Write $d(\noise,A):=\inf\{ |\noise-\noise'| \ | \, \noise'\in A \}$ for the Euclidean distance from $\noise$ to the set $A$.
Recognize the last sum as $\overline{\gmc}_{n,\varphi}(\noise')^{\otimes 2}[\kernel_n]$, use \eqref{e.morenoflores.A} for $\noise\mapsto\noise'$, and take the supremum of the resulting inequality over $\noise'\in A(\alpha)$.
Doing so gives the lower bound
\begin{align}
	\label{e.morenoflores.step1}
	\gmc_n(\noise)[\varphi] \geq \tfrac{1}{2} \base[\varphi] \, e^{-\sqrt{\alpha}\,d(\noise,A(\alpha))}\ .
\end{align}

Next, we derive a lower bound on $\PP[A(\alpha)]$.
Given that $\EE\,\gmc_n(\noise)[\varphi]=\base[\varphi]$, the Paley--Zygmund inequality yields
\begin{align}
	\label{e.morenoflores.A1bd}
	\PP[A_1]
	\geq
	\frac{\base[\varphi]^2}{4\EE\,(\gmc_n(\noise)[\varphi])^2}
	=
	\frac{\base[\varphi]^2}{4\base^{\otimes 2}_{\varphi}[e^{\kernel_n}]}
	=
	\frac{1}{4\rho_n}\ ,
\end{align}
where the first equality is readily checked from \eqref{e.gmck}.
Next, observing that the event $A_2(\alpha)^\compl \cap A_1$ implies $ \gmc_{n}^{\otimes 2}[\varphi^{\otimes 2} \kernel_n] \geq \alpha \base[\varphi]^2 / 4 $, we use Markov's inequality to write 
\begin{align}
	\label{e.morenoflores.A2bd}
	\PP[A_2(\alpha)^\compl \cap A_1]
	\leq
	\frac{4 \EE\,\gmc_{n}^{\otimes 2}[\varphi^{\otimes 2} \kernel_n]}{\alpha \base[\varphi]^2}
	=
	\frac{4\base^{\otimes 2}_{\varphi}[\kernel_ne^{\kernel_n}] }{ \alpha\base[\varphi]^2}
	=
	\frac{4\rho'_n}{\alpha}\ ,	
\end{align}
where again the first equality is readily checked from \eqref{e.gmck}.
Combining \eqref{e.morenoflores.A1bd}--\eqref{e.morenoflores.A2bd} and substituting $\alpha=2^{5}\rho_n\rho'_n$ give
\begin{align}
	\label{e.morenoflores.Abd}	
	\PP[A(\alpha_*)]
	\geq
	\frac{1}{4\rho_n} - \frac{4\rho'_n}{\alpha_*}
	=
	\frac{1}{2^3\rho_n}\ ,
	\qquad
	\alpha_* := 2^{5}\rho_n\rho'_n\ .
\end{align}

The proof is completed by using \eqref{e.morenoflores.step1} to obtain
$
	\PP[ \,\gmc_n(\noise)[\varphi] < e^{-r} \base[\varphi]\, ]
	\leq
	\PP[ d(\noise,A(\alpha_*)) \geq (r - \log 2)/\sqrt{\alpha_*}]
$ and
combining it with the Gaussian concentration inequality \cite[Lemma~2.2.11]{talagrand2003spin}
\begin{align}
	\PP \big[ d(\noise,A(\alpha_*)) \geq u + 2\sqrt{\log 2/\PP[A(\alpha_*)] } \, \big] \leq 2e^{-u^2/4}
\end{align}
for $u=((r-\log 2)/\sqrt{\alpha_*}-2\sqrt{\log 2^4\rho_n})_+$ with the aid of \eqref{e.morenoflores.Abd}.
\end{proof}

\bibliographystyle{alpha}
\bibliography{shf}

\end{document}